\documentclass[12pt,a4paper,twoside,reqno]{amsart}
\usepackage[utf8]{inputenc}
\usepackage[margin=0.9in]{geometry} 
\usepackage{amsthm}
\usepackage[dvipsnames]{xcolor}  
\usepackage{amsmath}
\usepackage{amsfonts}
\usepackage{amssymb}
\usepackage{chngcntr}
\usepackage{url}
\usepackage{color} 
\usepackage{hyperref}
\usepackage[makeroom]{cancel}
\usepackage{enumerate}
\usepackage{mathrsfs}
\newtheorem{Theorem}{Theorem}
\newtheorem{Conjecture}[Theorem]{Conjecture}

\newtheorem{Corollary}[Theorem]{Corollary}
\newtheorem{Lemma}[Theorem]{Lemma}

\newtheorem{Definition}[Theorem]{Definition}

\DeclareMathOperator*{\E}{\mathbb{E}}

\newcommand{\Var}{\operatorname{Var}}

\newcommand{\Po}{\operatorname{Po}}

\newcommand{\sgn}{\operatorname{sgn}}

\counterwithin{Theorem}{section}
\numberwithin{equation}{section}

\title{On the triple correlations of fractional parts of $n^2\alpha$} 
\RequirePackage[normalem]{ulem} 
\RequirePackage{color}\definecolor{RED}{rgb}{1,0,0}\definecolor{BLUE}{rgb}{0,0,1}

\providecommand{\DIFaddbegin}{} 
\providecommand{\DIFaddend}{}

\begin{document}

\author{Niclas Technau}
\address{A0048 - 480 Lincoln Dr.,
Madison, WI 53706-1325, USA}
\email{technau@wisc.edu}

\author{Aled Walker}
\address{Trinity College, Cambridge CB2 1TQ, United Kingdom}
\email{aw530@cam.ac.uk}

\thanks{While the work towards this paper was being carried out, NT was supported by the European Research Council (ERC) under the European Union's Horizon 2020 research and innovation programme 
(Grant agreement No. 786758).
AW is supported by a Junior Research Fellowship from Trinity College Cambridge, and was also supported by a Post-doctoral Fellowship 
at the Centre de Recherches Math\'{e}matiques. } 

\begin{abstract}
For fixed $\alpha \in [0,1]$, consider the set $S_{\alpha,N}$ of dilated squares 
$\alpha, 4\alpha, 9\alpha, \dots, N^2\alpha \, $ modulo $1$. Rudnick and Sarnak conjectured that for Lebesgue almost all such $\alpha$ the gap-distribution 
of $S_{\alpha,N}$ is consistent with the Poisson model
(in the limit as $N$ tends to infinity). 
In this paper we prove a new estimate for the triple correlations associated to this problem, establishing an asymptotic expression for the third moment of the number of elements of $S_{\alpha,N}$ in a random interval of length $L/N$, provided that $L > N^{1/4+\varepsilon}$. The threshold of $1/4$ is substantially smaller than the threshold of $1/2$ (which is the threshold that would be given by a na\"{i}ve discrepancy estimate). 

Unlike the theory of pair correlations, 
rather little is known about triple correlations 
of the dilations $(\alpha a_n \, \text{mod } 1)_{n=1}^{\infty} $ 
for a non-lacunary sequence $(a_n)_{n=1}^{\infty} $ of increasing integers. This is partially due to the fact that second moment of the triple correlation function is difficult to control, and thus standard techniques involving variance bounds are not applicable. We circumvent this impasse by
using an argument inspired by works of Rudnick--Sarnak--Zaharescu and Heath-Brown,
which connects the triple correlation function
to some modular counting problems. 

In an appendix we comment on the relationship between 
discrepancy and correlation functions, answering a question of Steinerberger.
\end{abstract}
\maketitle

\global\long\def\Rk{R_{k}}%
\global\long\def\Rkd{R_{k}'}%
\global\long\def\R3{R_{3}}%
\global\long\def\R3d{R_{3}'}%
\global\long\def\R2{R_{2}}%

 \section{Introduction}
 Let $(a_n)_{n=1}^\infty$ be a strictly increasing sequence of positive integers. This paper  is concerned with the distribution of $(\alpha a_n \, \text{mod } 1)_{n=1}^\infty$ in short intervals, for a generic dilate $\alpha \in [0,1]$, with particular focus on the case when $a_n = n^2$.

We begin with the familiar notion of the discrepancy $D_N$, defined to be 
\begin{equation}
\label{eq: discrepancy}
D_N = D_N((\alpha a_n \, \text{mod } 1)_{n=1}^\infty) = \sup_{0<a<b<1}  
\Big\vert \frac{\vert \{n\leq N: \alpha a_n \, \text{mod }1 \in (a,b)\} \vert}{N} 
- (b-a) \Big\vert.
\end{equation}
It is an old result of Weyl, contained in his 1916 paper \cite{W16}, 
that 
$ D_N((\alpha n^2 \, \text{mod } 1)_{n=1}^\infty) = o_{\alpha}(1)$ 
as $N\rightarrow \infty$ for any irrational $\alpha$. 
The sequence $(\alpha n^2 \, \text{mod } 1)_{n=1}^\infty$ 
is then said to be \emph{equidistributed} modulo $1$. 
One way of viewing Weyl's result is as demonstrating 
a pseudorandomness property for the sequence 
$(\alpha n^2 \, \text{mod } 1)_{n=1}^\infty$. 
Indeed, in the random model in which 
$( \alpha n^2 \, \text{mod } 1)_{n=1}^N$ is replaced by $N$ 
independent random variables  $X_1,\dots,X_N$ 
which are uniformly distributed on $[0,1)$ one has\footnote{This is by the law of the iterated logarithm. See Khintchine \cite{K24},
as well as Chung \cite[Theorem $2^*$]{Ch49} and Kuipers-Niederreiter \cite[p. 98]{KN12}.}  
$$
\limsup_{N\rightarrow \infty} 
\frac{ \sqrt{2N} D_N((X_n)_{n=1}^\infty)}{\sqrt{\log \log N}} = 1
$$ almost surely. In particular 
$\mathbb{E}D_N((X_n)_{n=1}^\infty) = o(1)$ as $N\rightarrow \infty$. 

One might wonder, considering the strong quantitative decay enjoyed in the random model, whether Weyl's result admits such a quantitative refinement. Unfortunately, as is well known, if $\alpha$ is well-approximated by rational numbers with small denominators then the discrepancy $D_N((\alpha n^2 \, \text{mod } 1)_{n=1}^\infty)$ can tend to zero extremely slowly. However, for a generic $\alpha$ the situation is much improved, and in fact for \emph{any} strictly increasing sequence of positive integers $(a_n)_{n=1}^\infty$ one has the classical result of Erd\H{o}s-Koksma \cite{EK49}, which implies that \[D_N((\alpha a_n \, \text{mod } 1)_{n=1}^\infty) = O_{\varepsilon}(N^{-\frac{1}{2} + \varepsilon})\] for Lebesgue almost all $\alpha \in [0,1]$. This trivially implies that, for almost all $\alpha \in [0,1]$, if $I \subset \mathbb{R}/\mathbb{Z}$ is a fixed interval of length $\vert I\vert \geqslant N^{-\frac{1}{2} + \varepsilon}$ then 
\begin{equation}
\label{eq discrepancy estimate}
\vert \{ n \leqslant N: \alpha n^2 \, \text{mod } 1 \in I\}\vert = (1 +o_{\varepsilon}(1)) N \vert I\vert.
\end{equation}
So, at least for a generic $\alpha$, pseudorandomness is enjoyed down to the scale $N^{-\frac{1}{2} + \varepsilon}$. A result of Aistleitner and Larcher \cite[Cor. 1]{AL16} 
implies that the exponent 
$1/2$ is optimal for the metric discrepancy problem: for generic $\alpha$ and for each $\varepsilon >0$
there are infinitely many $N$ such that 
$ D_N((\alpha a_n \, \text{mod } 1)_{n=1}^\infty) > N^{-1/2 -\varepsilon}$, provided $a_n = P(n)$ 
for some polynomial $P$, of degree at least two,
with integer coefficients. 

By allowing the interval $I$ to vary, one can develop a related notion of pseudorandomness at scales that are smaller than $N^{-1/2}$, one which concerns the `clustering' of the points $\alpha n^2 \, \text{mod } 1$. To introduce this notion, which will be the main focus of the paper, we let $Y$ be a random variable that is uniformly distributed on $[0,1)$, and for a natural number $N$ and a parameter $L$ in the range $0 < L \leqslant N$ we let $W_{\alpha, L,N}$ be the random variable 
\begin{equation}
\label{eq: W definition}
W_{\alpha, L,N}:=\vert \{n \leqslant N :\alpha n^2\in [Y,Y+L/N] \, \text{mod } 1\}\vert.
\end{equation} 
It is easy to see that $\mathbb{E} W_{\alpha,L,N} = L$. But how should one expect $W_{\alpha,L,N}$ to be distributed in the limit $N \rightarrow \infty$ (for a generic dilate $\alpha$)? Consider the same random model as before, in which $( \alpha n^2 \, \text{mod } 1)_{n=1}^N$ is replaced by $N$ independent random variables  $X_1,\dots,X_N$ which are uniformly distributed on $[0,1)$. Then, if $L$ is constant as $N \rightarrow \infty$, letting 
\begin{equation}
\label{random model}
Z_{L,N} :=  \vert\{n \leqslant N : X_n \in [Y,Y+L/N] \, \text{mod } 1 \}\vert
\end{equation} one may calculate that \[ Z_{L,N} \xrightarrow{dist} \Po(L) \] as $N \rightarrow \infty$, where $\Po(L)$ is a Poisson-distributed random variable with parameter $L$.

Having described this random model, we can now state the following remarkable conjecture:

\begin{Conjecture}[Rudnick-Sarnak\footnote{These authors refer to the `distribution of the spacings between the elements', rather than mentioning the random variables $W_{\alpha,L,N}$ directly, but these are equivalent notions. For more on this other perspective, see the introduction to \cite{RSZ01}.} \cite{RS98}]
\label{Conjecture RS}
For almost all $\alpha \in [0,1]$, for all fixed $L >0$, \[ W_{\alpha,L,N} \xrightarrow{dist} \Po(L)\] as $N \rightarrow \infty$. 
\end{Conjecture}

\noindent If true, this conjecture would represent a strong local notion of pseudorandomness for the sequence $(\alpha n^2 \, \text{mod } 1)_{n=1}^\infty$, at least for a generic $\alpha$. In fact, a further conjecture \cite[p. 38]{RSZ01} posits more information about the full-measure set of suitable dilates $\alpha$. To state this conjecture, we recall that $\alpha$ is of type $\omega$ if there are only finitely many pairs $(a,q)$ with $\vert \alpha - a/q\vert < q^{-\omega}$.
\begin{Conjecture}[Rudnick--Sarnak--Zaharescu]
\label{Conjecture RSZ}
If $\alpha$ is of type $2 + \varepsilon$ for all $\varepsilon >0$, and the convergents $a/q$ to $\alpha$ satisfy \[\lim\limits_{ q \rightarrow \infty} \log \widetilde{q} / \log q = 1,\] where $\widetilde{q}$ is the square-free part of $q$, then for all fixed $L>0$, as $N \rightarrow \infty$, one has\[ W_{\alpha,L,N} \xrightarrow{dist} \Po(L).\]
\end{Conjecture}

\noindent \emph{Remark}: We remind the reader that, by Dirichlet's approximation theorem, 
the type of a real number is never less than $2$.
Further, by Khintchine's theorem, a generic number is of type
$2+\varepsilon$ for all $\varepsilon >0$.
One can also readily find explicit examples, like $\alpha = \sqrt{2}$, by using continued fractions. \\

Conjectures \ref{Conjecture RS} and \ref{Conjecture RSZ} appear to lie very deep. One hypothetical approach for showing the desired convergence in distribution would be to use the method of moments. More precisely, if one could show, for a generic $\alpha$, for a random variable $X_L \sim \Po(L)$, and for all $k \in \mathbb{N}$, that $\E W_{\alpha, L,N}^k \rightarrow \E X_L^k$ as $N \rightarrow \infty$ then Conjecture \ref{Conjecture RS} would follow. Rudnick--Zaharescu \cite{RZ02} used this approach to show that if $(a_n)_{n=1}^\infty$ is a lacunary sequence of natural numbers then for almost all $\alpha$ the sequence $(\alpha a_n \, \text{mod } 1)_{n=1}^\infty$ has spacing statistics that agree with the Poisson model. For the squares, the first non-trivial case is $k = 2$, and this was settled by Rudnick--Sarnak\footnote{We should remark that Rudnick--Sarnak phrased their result in terms of the pair correlation function, rather than explicitly mentioning $W_{\alpha,L,N}$, but the results are equivalent.} some 22 years ago.  
\begin{Theorem}[Rudnick--Sarnak \cite{RS98}]
\label{Theorem RS}
For almost all $\alpha \in [0,1]$,for all fixed $L >0$, \[ \mathbb{E} W_{\alpha,L,N}^2 = L + L^2 + o_{\alpha,L}(1)\] as $N \rightarrow \infty$. 
\end{Theorem}
\noindent This gives an estimate on the so-called \emph{number variance} $\Var(W_{\alpha,L,N})$, namely \[ \Var(W_{\alpha,L,N}) = L +o_{\alpha, L}(1)\] as $N \rightarrow \infty$.  

Very little is known regarding the higher moments, although certain results can be extracted from the literature. For larger $L$,  the issue is settled by the aforementioned discrepancy bounds. Indeed, expression (\ref{eq discrepancy estimate}) implies that for almost all $\alpha \in [0,1]$, for all integers $k \geqslant 1$, for all $N \in \mathbb{N}$ and for all $L \in \mathbb{R}$ in the range $ N^{1/2 + \varepsilon} \leqslant L \leqslant N$, 
\begin{equation}
\label{eq: consequence of discrepancy}
\mathbb{E} W_{\alpha,L,N}^k = L^k(1 + o_{\alpha,\varepsilon,k}(1))
\end{equation}
\noindent as $N \rightarrow \infty$ (where the error term is independent of the choice of parameters $L$). We will give the simple proof of (\ref{eq: consequence of discrepancy}) in Appendix \ref{Appendix B}, alongside other consequences of the discrepancy bounds. In passing, we will answer a question of Steinerberger from \cite{S18}. 

One might wonder whether the methods that Rudnick--Sarnak introduced to tackle the second moment in Theorem \ref{Theorem RS} could be applied to higher moments. Unfortunately, in the case when $L$ is constant, Rudnick--Sarnak already noted in \cite[Section 4]{RS98} that their method faces a major obstruction when applied to higher moments. We will describe this obstruction in Appendix \ref{Appendix C}, where (for the third moment) we observe that the obstruction persists for $L$ as large as $N^{1/3}$. \\

We now present the main result of this paper.
\begin{Theorem}[Main Theorem]
\label{Theorem TW}
Let $\varepsilon>0$. Then, for almost all $\alpha \in [0,1]$, for all $N \in \mathbb{N}$ and for all $L \in \mathbb{R}$ in the range $N^{1/4 + \varepsilon} < L \leqslant N$ we have \[ \mathbb{E} W_{\alpha,L,N}^3 = L^3(1 + o_{\alpha, \varepsilon}(1))\] as $N\rightarrow \infty$, where the $o_{\alpha,\varepsilon}(1)$ term is independent of the choice of the parameters $L$. 
\end{Theorem}
\noindent Note in particular that $1/4 < 1/3$, so we successfully give an asymptotic expression for the third moment in part of the range in which the Rudnick--Sarnak obstruction holds (see Appendix \ref{Appendix C}). \\

To prove Theorem \ref{Theorem TW}, we will first perform a standard reduction to a statement concerning correlation functions. These functions are closely related to the moments $\mathbb{E} W_{\alpha,L,N}^k$, but they can be more convenient to analyse.  

\begin{Definition}[Correlation functions]
\label{Definition correlation function}
Let $\alpha \in [0,1]$, let $k \geqslant 2$ be a natural number, and let $g:\mathbb{R}^{k-1} \rightarrow [0,1]$ be a compactly supported function. Then the $k^{th}$ correlation function $R_k(\alpha,L,N, g)$ is defined to be \[\frac{1}{N}\sum\limits_{\substack{1 \leqslant x_1, \dots, x_k \leqslant N \\
\text{distinct}}} g\Big(\frac{N}{L} \{ \alpha (x_1^2 - x_2^2)\}_{\sgn}, \frac{N}{L}\{ \alpha (x_2^2 - x_3^2)\}_{\sgn}, \dots, \frac{N}{L}\{ \alpha (x_{k-1}^2 - x_k^2)\}_{\sgn}\Big),\] where \[ \{ \cdot \}_{\sgn}:\mathbb{R} \longrightarrow (-1/2,1/2]\] denotes the signed distance to the nearest integer, and $N$ and $L$ are real parameters. 
\end{Definition}
\noindent In practice, one only needs to understand the correlation functions when the test function $g$ is sufficiently nice, e.g. smooth or the indicator function of a box. 

When proving Theorem \ref{Theorem RS}, Rudnick--Sarnak manipulated the pair correlation function $R_2(\alpha,L,N,g)$. We manipulate the triple correlation function $R_3(\alpha,L,N,g)$, proving the following result, which, for readers more familiar with correlation functions than with the moments of $W_{\alpha,L,N}$, might seem to be more natural. 

\begin{Theorem}[Triple correlations]
\label{Theorem main theorem}
Let $\varepsilon >0$. Then, for almost all $\alpha \in [0,1]$, for all compactly supported continuous functions $g: \mathbb{R}^2 \longrightarrow [0,1]$, for all $N \in \mathbb{N}$ and for all $L \in \mathbb{R}$ in the range $N^{1/4 + \varepsilon}<L<N^{1- \varepsilon}$ we have\[ R_3(\alpha,L,N,g) = (1 + o_{\alpha,\varepsilon, g}(1))L^2 \Big(\int g(\mathbf{w}) \, d\mathbf{w}\Big)\] as $N \rightarrow \infty$, where the $o_{\alpha,\varepsilon,g}(1)$ term is independent of the choice of the parameters $L$. 
\end{Theorem}

Before continuing to survey other relevant papers, we should stop to explain why the spacing statistics of the sequence $\alpha n^2 \, \text{mod } 1$ are of a particular interest (aside from as a part of the larger endeavour of finding pseudorandomness in arithmetic sequences). This is due to a connection between number theory and theoretical physics known as \emph{arithmetic quantum chaos}. In brief, the spacing statistics between elements of the sequence of $\alpha n^2 \, \text{mod } 1$ correspond to the spacing statistics between the eigenvalues of a certain quantum system. (This system is a two-dimensional boxed oscillator, with a harmonic potential in one direction and hard walls in the other, as described in the introduction to \cite{RS98}.) A famous and far-reaching observation of Berry--Tabor \cite{BT77} then suggests that such spacing statistics in the semi-classical limit (i.e. the distribution of $W_{\alpha,L,N}$ for constant $L$ as $N \rightarrow \infty$) should be determined by the dynamics of the corresponding classical system. Regular (integrable) classical dynamics should correspond to spacing statistics in asymptotic agreement with the Poisson model. 

Unfortunately if $\alpha$ is rational (or is very well approximated by rationals with square denominators), it is easy to prove that high moments of $W_{\alpha,L,N}$ do not agree with the Poisson model (see \cite[Theorem 2]{RSZ01})! However, excluding such zero-measure counter-examples\footnote{Zaharescu \cite{Za03} showed that in a precise sense that, at least amongst all very well approximable $\alpha$, these were the only counter-examples.}, one might still hope for a metric result. 

There are precious few examples of fixed sequences 
for which full information is known about the spacing statistics. For the sequence $(\sqrt{n} \, \text{mod } 1)_{n=1}^{\infty}$ Elkies and McMullen \cite{EM04} have established, using dynamical methods,
the gap distribution of 
$(\sqrt{n} \, \text{mod } 1)_{n=1}^{\infty}$ (which is not from a Poisson model); El-Baz, Marklof, and Vinogradov \cite{BMV15} have
demonstrated that the second moment of this gap distribution is nonetheless in accordance with the Poisson model. Further, there are the results in \cite{M03} and \cite{EMM05} on the second moment of the gap distribution between values of quadratic forms. However, most of the results in the literature are metric in at least one parameter (e.g. \cite{RS98}, \cite{RZ02},\cite{Sa96}, \cite{ALL17}), and such results still have substantial content.

To delve further into the relationship to theoretical physics and to other spacing statistics would be to digress too far from our main theme; we direct the interested reader to the articles of Marklof \cite{M01} and Rudnick \cite{R08} for more on these issues. \\

Returning to the study of the moments of $W_{\alpha,L,N}$, and the discussion of relevant work, we continue with the paper \cite{RSZ01}. Here Rudnick--Sarnak--Zaharescu develop tools to relate the diophantine approximation properties of $\alpha$ to the size of the moments $\mathbb{E} W_{\alpha,L,N}^k$. The main result of that paper can be phrased as follows:

\begin{Theorem}[Rudnick-Sarnak-Zaharescu] \label{Theorem RSZ}
Let $\alpha \in [0,1]$, and suppose that there are infinitely many rationals $b_j/q_j$,
with $q_j$ prime, satisfying 
\begin{equation}
\Big\vert \alpha - \frac{b_j}{q_j}\Big\vert < \frac{1}{q_j^3}
\end{equation}
Then there is a subsequence $N_j \rightarrow \infty$, with $ \log N_j/
\log q_j
\rightarrow 1$ for which, for all $L>0$, \[ W_{\alpha,L,N_j} \xrightarrow{dist} \Po(L)\] as $j \rightarrow \infty$. 
\end{Theorem} 
\noindent 
The authors of \cite{RSZ01} sacrificed the genericness of $\alpha$ (working instead with those $\alpha$ which are unusually well-approximable) in favour of control over the moments $\mathbb{E} W_{\alpha,L,N_j}^k$ for constant $L$ and for all $k$. One may switch objectives in their analysis, sacrificing the range of $L$ in order to work with almost all $\alpha$. Applying this switch in the context of the third moment calculation, their method shows the following result:

\begin{Theorem}[Method of R--S--Z]
\label{Theorem method of RSZ}
Let $\varepsilon>0$. Then, for almost all $\alpha \in [0,1]$, for all $N \in \mathbb{N}$ and for all $L \in \mathbb{R}$ in the range $N^{3/5 + \varepsilon} < L \leqslant N$ we have\[ \mathbb{E} W_{\alpha,L,N}^3 = L^3(1 + o_{\alpha,\varepsilon}(1))\] as $N \rightarrow \infty$, where the $o_{\alpha,\varepsilon}(1)$ term is independent of the choice of parameters $L$. Moreover, one can give an explicit description of a suitable full-measure set of suitable $\alpha$ (in terms of properties of rational approximations to $\alpha$). 
\end{Theorem}
\noindent Although the range of $L$ in Theorem \ref{Theorem method of RSZ} is much smaller than the range in our result, the method of R--S--Z yields a more explicit description of the full-measure set of suitable $\alpha$. We will indicate how to extract Theorem \ref{Theorem method of RSZ} from \cite{RSZ01} in Section \ref{section outline of proof} below.

Our approach to proving Theorem \ref{Theorem main theorem} is  inspired by this work of R--S--Z, but also by the work of Heath-Brown in \cite{HB10}, who introduced a related technique for studying the pair correlation function $R_2(\alpha, L,N,g)$. The full description of the method will come in Section \ref{section outline of proof}, but we sketch the idea here, so as to explain in a rough way how we extract an improvement over \cite{RSZ01}. After having replaced $\alpha$ by a rational approximation $a/q$, one transforms the triple correlation function into an expression that counts the number of solutions to certain polynomial equations modulo $q$. The equations which occur are of the form
\begin{equation}
\label{exposition equation}
 \{ x,y,z \leqslant N: x^2 - y^2 \equiv c_1 \,(\text{mod } q), y^2 - z^2 \equiv c_2 \, (\text{mod } q)\},
\end{equation} for certain ranges of $N$ and $q$ and for certain sets of coefficients $c_1$ and $c_2$. In \cite{RSZ01} the number of solutions was estimated by using the `completion of sums' technique to remove the $N$ cut-off, followed by an implementation of the Hasse-Weil bound. In our work we manage to take advantage of the extra averaging over $c_1$ and $c_2$ which is present in the problem, together with some intricate (though elementary) exponential sum arguments, which ends up leading to a stronger bound for certain ranges of $N$ and $q$. Heath-Brown did something similar for the pair correlation function \cite{HB10}, but the analysis of the relevant exponential sums for the triple correlations is substantially more delicate. 

Another relevant work is the paper of Kurlberg and Rudnick \cite{KR99}, in which those authors established that the spacing of quadratic residues mod $q$ as the number of prime factors of $q$ grows is in agreement with the Poisson model. Lemma \ref{lem: approx size of A_not} below could be viewed as a special case of the arguments of that paper. However, as will become evident, the work here necessarily concerns a rather sparse subset of the set of quadratic residues mod $q$, as $N \approx q^{\theta}$ with $\theta < 1$, and so  the work of \cite{KR99} is not directly applicable. Pair correlations of rational functions mod $q$ were also studied by Boca and Zaharescu \cite{BZ00}, but again, we will not be able to use that paper directly. 

Very recently, and after the first version of the present manuscript was submitted, Lutsko released a preprint \cite{Lu20} which generalised Theorem \ref{Theorem main theorem} to all long-range correlation functions $R_k(\alpha,L,N,g)$, provided $L \geqslant N^{\frac{k-2}{2k-2} + \varepsilon}$. When $k=3$, the threshold $\frac{k-2}{2k-2}$ recovers our threshold of $N^{1/4}$ for triple correlations. Note also that $\frac{k-2}{2k-2} \rightarrow 1/2$ from below as $k \rightarrow \infty$, i.e. Lutsko's threshold approaches the trivial discrepancy bound for large $k$. The methods of \cite{Lu20} seem to be completely different to our own, and are instead based on an analysis of the sums $\sum_{n \leqslant N} e(\alpha m n^2)$ using the van der Corput B process (and subsequent stationary phase estimates). It remains to be seen whether any improvement to the threshold $c_3 = 1/4$ could be derived by combining these two different approaches.\\

The structure of the paper is as follows. In Section \ref{Section reduction to correlation functions} we will give the standard argument (passing from moments to correlation functions) which reduces Theorem \ref{Theorem TW} to Theorem \ref{Theorem main theorem}. The proof of Theorem \ref{Theorem main theorem} is then given in Section \ref{section outline of proof}, in which it is resolved subject to three auxiliary results (one concerning diophantine approximation, the other two concerning the number of solutions to certain diophantine equations similar to (\ref{exposition equation})). The final three sections of the paper settle these results -- one per section -- thus concluding the main proof. 

The appendices contain some arguments that are minor modifications of the literature (but which are nonetheless pertinent to the main paper). These are, respectively, a version of Theorem \ref{Theorem RS} in which $L$ grows with $N$; the proof of the asymptotic (\ref{eq: consequence of discrepancy}); and the discussion of the obstruction to the study of triple correlations that was identified by Rudnick-Sarnak. \\

\noindent\textbf{Acknowledgements}: We would like to thank Zeev Rudnick and Christoph Aistleitner for helpful comments relating to previous versions of the manuscript, and Andrew Granville and Dimitris Koukoulopoulos for many interesting conversations. Thanks also to several anonymous referees for their suggestions and corrections. \\

\noindent \textbf{Notation}: Most of our notation is standard, but perhaps we should highlight a few conventions. For a natural number $q$ we let $e_q(x)$ be a shorthand for $e^{2 \pi i x/q}$, and given a parameter $M\geqslant 1$ we let $[M]$ denote the set $\{m \in \mathbb{N}: 1 \leqslant m \leqslant M\}$. In particular $1_{[M]}$ denotes the indicator function of all the natural numbers at most $M$. If a range of summation is given as $\sum_{m \leqslant M}$ then it is assumed that $m$ is a natural number and that $m \geqslant 1$. Finally, for $x \in \mathbb{R}$, we will use $\Vert x\Vert$
to denote the distance from $x$ to the nearest integer,
and $\{x\}_{\sgn}$ to denote the signed distance to the nearest integer from $x$.\\

\section{Reduction to correlation functions}
\label{Section reduction to correlation functions}
In this short section we will reduce Theorem \ref{Theorem TW} to Theorem \ref{Theorem main theorem}. Firstly, since estimate (\ref{eq: consequence of discrepancy}) holds for large $L$ we may assume without loss of generality that $L < N^{1- \varepsilon}$. Then we use linearity of expectation to deduce that
\begin{align}
\label{three terms}
\mathbb{E}W_{\alpha,L,N}^3 &= \sum\limits_{x,y,z \leqslant N} \mathbb{P}(\alpha x^2, \alpha y^2, \alpha z^2 \in [Y,Y + L/N] \, \text{mod } 1) \nonumber\\
& = \sum\limits_{x \leqslant N} \mathbb{P}(\alpha x^2 \in [Y,Y+L/N] \, \text{mod } 1) + 3 \sum\limits_{ \substack{ x,y \leqslant N \\ x \neq y}} \mathbb{P}(\alpha x^2, \alpha y^2 \in [Y,Y+L/N] \, \text{mod } 1) \nonumber \\&+ \sum\limits_{ \substack{ x,y,z, \leqslant N \\ \text{distinct}}} \mathbb{P}(\alpha x^2, \alpha y^2 , \alpha z^2 \in [Y,Y+L/N] \, \text{mod } 1),
\end{align}
\noindent where $Y$ is a random variable that is uniformly distributed modulo $1$. The first of the three terms in (\ref{three terms}) is equal to $L$, so may be absorbed into the error term of Theorem \ref{Theorem TW}. The second term is 
\begin{align}
\label{second term}
=& 3 \sum\limits_{ \substack{ x,y \leqslant N \\
 x \neq y\\
\Vert \alpha(x^2-y^2)\Vert\leqslant L/N}} \Big( \frac{L}{N} - \Vert \alpha (x^2 - y^2)\Vert\Big) \nonumber\\
=& 3L \Big(\frac{1}{N} \sum\limits_{ \substack{x,y, \leqslant N \\ x \neq y}} \max \Big(0, 1 - \frac{N}{L}\Vert \alpha(x^2 - y^2)\Vert\Big)\Big)\nonumber \\
=& 3L R_2(\alpha,L,N,f),
\end{align}
where $f:\mathbb{R} \longrightarrow [0,1]$ is the function $f(x) = \max(0,1 - \vert x\vert)$ and $R_2(\alpha,L,N,f)$ is the pair correlation function as defined in Definition \ref{Definition correlation function}. In Appendix \ref{Appendix A} we will show, by a trivial adaptation of the known techniques, that for almost all $\alpha \in [0,1]$ one has
\begin{equation}
\label{pair correlation estimate that we need}
R_2(\alpha,L,N,f) = (1 + o_{\alpha,f}(1))L \Big(\int f(x) \, dx\Big) 
= (1 + o_{\alpha, f}(1))L. 
\end{equation} Therefore expression (\ref{second term}) is equal to $3L^2 + o_{\alpha,f}(L^2)$, which may also be absorbed into the error term of Theorem \ref{Theorem TW}. 

What remains is the third term of (\ref{three terms}). This is equal to 
\begin{equation}
\label{beginning of triple correlation}
\frac{L}{N} \sum\limits_{ \substack{ x,y,z \leqslant N \\ \text{distinct}}}\max(0,( 1 - \frac{N}{L}\max(\Vert \alpha(x^2 - y^2)\Vert, \Vert \alpha(y^2 - z^2)\Vert, \Vert \alpha(z^2 - x^2)\Vert)).
\end{equation}
Since $(x^2 - y^2) + (y^2 - z^2) = x^2 - z^2$ we see that (\ref{beginning of triple correlation}) is equal to an expression of the form $R_3(\alpha,L,N,g)$ for some continuous compactly supported function $g:\mathbb{R}^2 \longrightarrow [0,1]$. Indeed, \[ g(w_1,w_2) = \begin{cases}
 \max(0, 1 - w_1 - w_2) & w_1,w_2 \geqslant 0 \\
 \max(0, 1 - \max(w_1, -w_2)) & w_1 \geqslant 0, \, w_2 \leqslant 0 \\
 \max(0, 1 - \max(-w_1,w_2)) & w_1 \leqslant 0, \, w_2 \geqslant 0 \\
 \max(0, 1 + w_1 + w_2) & w_1, w_2 \leqslant 0. \end{cases} \]
An elementary calculation then demonstrates that \[ \int\limits_{-\infty}^{\infty} \int\limits_{-\infty}^{\infty} g(w_1,w_2) \, dw_1\, dw_2 = 1.\] Therefore, by Theorem \ref{Theorem main theorem}, if $L > N^{\frac{1}{4} + \varepsilon}$ then for almost all $\alpha \in [0,1]$ expression (\ref{beginning of triple correlation}) is equal to 
$L^3(1 + o_{\alpha,\varepsilon}(1))$ as $N \rightarrow \infty$. So Theorem \ref{Theorem TW} is proved. \qed \\

We make the usual remark that, by approximating the continuous function $g$ above and below by step functions, to prove Theorem \ref{Theorem main theorem} it will be enough to prove the following result: 

\begin{Theorem}
\label{Theorem main theorem step function}
Let $\varepsilon >0$. Then for almost all $\alpha \in [0,1]$, for all $s_1,t_1,s_2,t_2 \in \mathbb {R}$ for which $s_1 < t_1$ and $s_2 < t_2$, and for all $L$ in the range $N^{1/4 + \varepsilon}<L<N^{1- \varepsilon}$, we have \[ R_3(\alpha,L,N,g_{\mathbf{s},\mathbf{t}}) = (1 + o_{\alpha,\varepsilon, \mathbf{s},\mathbf{t}}(1))L^2 (t_1 - s_1)(t_2 - s_2)\] as $N\rightarrow \infty$, where $\mathbf{s} = (s_1,s_2)$, $\mathbf{t} = (t_1,t_2)$, and $g_{\mathbf{s},\mathbf{t}}$ is the indicator function of the box $[s_1,t_1] \times [s_2,t_2]$. \\
\end{Theorem}

It will turn out to be crucial in our subsequent methods that the ratio $\log L/\log N$ does not vary too wildly. To finish this section, we will show how to deduce Theorem \ref{Theorem main theorem step function} from the following weaker result:

\begin{Theorem}
\label{Theorem main theorem step function constrained}
Let $\beta \in (0, 3/4)$ and let $\eta >0$. Then, if $\eta\max(\beta^{-1}, (3/4 - \beta)^{-1})$ is small enough, the following holds: for almost all $\alpha \in [0,1]$, for all $s_1,t_1,s_2,t_2 \in \mathbb {R}$ for which $s_1 < t_1$ and $s_2 < t_2$, for all $N \in \mathbb{N}$ and for all $L \in \mathbb{R}$ in the range $N^{1 - \beta - \eta}<L<N^{1- \beta + \eta}$, we have \[ R_3(\alpha,L,N,g_{\mathbf{s},\mathbf{t}}) = (1 + o_{\alpha,\beta, \eta, \mathbf{s},\mathbf{t}}(1))L^2 (t_1 - s_1)(t_2 - s_2)\] as $N\rightarrow \infty$, where $\mathbf{s} = (s_1,s_2)$, $\mathbf{t} = (t_1,t_2)$, and $g_{\mathbf{s},\mathbf{t}}$ is the indicator function of the box $[s_1,t_1] \times [s_2,t_2]$. The error term is independent of the exact choice of the parameters $L$. 
\end{Theorem}
\begin{proof}[Deduction of Theorem \ref{Theorem main theorem step function} from Theorem \ref{Theorem main theorem step function constrained}]
Let $\varepsilon >0$ and choose $\eta>0$ such that $\eta \varepsilon^{-1}$ is suitably small. Let $\{\beta_1,\dots, \beta_R\}$ be a maximal $\eta$-separated subset of $[\varepsilon,3/4 - \varepsilon]$. Then for each $\beta_i$ we get a full measure set $\Omega_i \subset [0,1]$ such that, if $\alpha \in \Omega_i$, the conclusion of Theorem \ref{Theorem main theorem step function constrained} holds with $\beta = \beta_i$ and with $\eta$ as chosen. We claim that $\Omega = \cap_{i \leqslant R} \Omega_i$ is a suitable full measure set of values of $\alpha$ in Theorem \ref{Theorem main theorem step function}. 

Indeed, let $\alpha \in \Omega$ and let $s_1,t_1,s_2,t_2 \in \mathbb{R}$ with $s_1 < t_1$ and $s_2 < t_2$. From Theorem \ref{Theorem main theorem step function constrained} we know that for all $i \leqslant R$, for any $\delta>0$, for all $N \geqslant N_0(\alpha,\beta_i,\eta, \delta, \mathbf{s},\mathbf{t})$, and for any $L$ in the range $N^{1 - \beta_i - \eta} < L < N^{1 - \beta_i + \eta}$, we have \begin{equation}
\label{eq:delta correlation}
\vert R_3(\alpha,L,N,g_{\mathbf{s},\mathbf{t}}) - L^2(t_1 - s_1)(t_2 - s_2)\vert < \delta  L^2.
\end{equation} Now, given any $N$ and any $L$ in the range $N^{1/4 + \varepsilon} < L < N^{1 - \varepsilon}$, there exists an $i$ such that $N^{1 - \beta_i - \eta} < L < N^{1 - \beta_i + \eta}$. Therefore, if $N \geqslant \max_{i\leqslant R} N_0(\alpha,\beta_i,\eta, \delta, \mathbf{s},\mathbf{t})$, the inequality (\ref{eq:delta correlation}) holds. Since $\delta$ is arbitrary, the conclusion of Theorem \ref{Theorem main theorem step function} holds. \end{proof}
\vspace{3mm}
\section{Proof of Theorem \ref{Theorem main theorem step function constrained}}
\label{section outline of proof}
Our task is now to prove Theorem \ref{Theorem main theorem step function constrained}. Let us fix $\beta$ and $\eta$, which is assumed to be sufficiently small, and for the time being let us also fix some $\alpha \in [0,1]$ and some $s_1,t_1,s_2,t_2 \in \mathbb{R}$ with $s_1 < t_1$ and $s_2 < t_2$. We may also assume, without loss of generality, that $N$ is sufficiently large in terms of $\alpha$, $\beta$, $\eta$, $\mathbf{s}$ and $\mathbf{t}$. 

We begin by replacing $\alpha$ with a suitably good
rational approximation $a/q$. Assume that there exists some rational $a/q$, with $q$ prime, for which
\begin{equation}
\label{eq: diophantine approx} \Big
\vert \alpha -\frac{a}{q} \Big\vert 
\leqslant \frac{1}{q^{2-\eta}} 
\end{equation} and
\begin{equation}
\label{equation relationship of N and q}
N^{\frac{2 + \beta}{2} + 10\eta} \leqslant q \leqslant 2 N^{\frac{2 + \beta}{2} + 10\eta}.
\end{equation}
\noindent We will show in Section \ref{Section: approximation with prime denominator} that almost all $\alpha$ admit such an approximation. Then, for such a pair $(N,q)$, we define \begin{equation}
\label{equation definition of ANqc_1c_2}
A(N,q,c_1,c_2): = \vert \{ x,y,z \leqslant N : x^2 - y^2 \equiv c_1 \, (\text{mod } q), \, y^2 - z^2 \equiv c_2 \, (\text{mod } q)\}\vert.
\end{equation} We then claim that \begin{align}
\label{equation sandwiching key}
\frac{1}{N}\sum\limits_{\substack{(r_1,r_2) \in S^- \\ r_1r_2 \neq 0 \\ r_1 + r_2 \neq 0}} 
 A(N,q,\overline{a} r_1,\overline{a} r_2) 
&\leqslant 
R_3(\alpha, L,N,g_{\mathbf{s},\mathbf{t}}) \leqslant \frac{1}{N}
\sum\limits_{\substack{(r_1,r_2) \in S^+ \\ r_1r_2 \neq 0 \\ r_1 + r_2 \neq 0}}  A(N,q,\overline{a} r_1,\overline{a} r_2),
\end{align}
\noindent where $\overline{a}$ denotes the inverse of $a$ modulo $q$,
\begin{equation}
\label{eq S minus}
S^{-} = \Big \{ (r_1,r_2) \in \mathbb{Z}^2: \frac{s_i qL}{N} + \frac{N^2}{q^{1- \eta}} \leqslant r_i \leqslant \frac{t_i qL}{N} - \frac{N^2}{q^{1- \eta}}, \, i = 1,2
\Big\},
\end{equation} and 
\begin{equation}
\label{eq S plus}
S^+ = \Big \{ (r_1,r_2) \in \mathbb{Z}^2: \frac{s_i qL}{N} - \frac{N^2}{q^{1- \eta}} \leqslant r_i \leqslant \frac{t_i qL}{N} + \frac{N^2}{q^{1- \eta}}, \, i=1,2
\Big \}.
\end{equation}

Indeed, given $x,y$ in the range $1 \leqslant x,y \leqslant N$ let us consider $r_1 \in \mathbb{Z}$ to be defined by the relation \[ a (x^2 - y^2) \equiv r_1 \, (\text{mod }q)\] and $-q/2 < r_1 < q/2$. Suppose that \[ \frac{s_1 qL}{N} + \frac{N^2}{q^{1- \eta}} \leqslant r_1 \leqslant \frac{t_1 qL}{N} - \frac{N^2}{q^{1- \eta}}.\] Then we have the inequalities
\begin{align*}
\{\alpha(x^2 - y^2) \}_{\sgn} \leqslant \{\frac{a}{q} (x^2 - y^2)\}_{\sgn} + \Big\vert \Big( \alpha - \frac{a}{q}\Big)(x^2 - y^2)\Big\vert  \leqslant \frac{r_1}{q} + \frac{N^2}{q^{2-\eta}} \leqslant  t_1 \frac{L}{N},
\end{align*}
\noindent and \[\{\alpha(x^2 - y^2) \}_{\sgn} \geqslant \{\frac{a}{q} (x^2 - y^2)\}_{\sgn} -\Big\vert \Big( \alpha - \frac{a}{q}\Big)(x^2 - y^2)\Big\vert  \geqslant \frac{r_1}{q} - \frac{N^2}{q^{2-\eta}} \geqslant  s_1 \frac{L}{N}.\] Suppose instead that \[ s_1 \frac{L}{N} \leqslant \{ \alpha(x^2 - y^2)\}_{\sgn} \leqslant t_1 \frac{L}{N}.\] Then, similarly to the above, we have
\[r_1 = q \{\frac{a}{q}(x^2 - y^2)\}_{\sgn} \leqslant q\Big(\{ \alpha(x^2 - y^2)\}_{\sgn} + \Big\vert \Big( \alpha - \frac{a}{q}\Big)(x^2 - y^2)\Big\vert\Big) \leqslant t_1 \frac{Lq}{N} + \frac{N^2}{q^{1- \eta}} \] and \[r_1 = q \{\frac{a}{q}(x^2 - y^2)\}_{\sgn} \geqslant q\Big(\{ \alpha(x^2 - y^2)\}_{\sgn} - \Big\vert \Big( \alpha - \frac{a}{q}\Big)(x^2 - y^2)\Big\vert\Big) \geqslant s_1 \frac{Lq}{N} - \frac{N^2}{q^{1- \eta}} . \] 

Finally, take $1 \leqslant z \leqslant N$ and define $r_2$ by the relation \[ a(y^2 - z^2) \equiv r_2 \, (\text{mod } q)\] with $-q/2 < r_2 < q/2$. Then, since $N < q/2$, we have that $x,y,z$ are distinct if and only if $r_1r_2 \neq 0$ and $r_1 + r_2 \neq 0$. 

From all these observations taken together, claim (\ref{equation sandwiching key}) is settled. \\

\noindent \emph{Remark}: The reader might find it helpful to note, at this early stage, that the relative sizes of $N$ and $q$ were chosen so that the two terms $qL/N$ and $N^2/q^{1-\eta }$ which appear in (\ref{eq S minus}) and (\ref{eq S plus}) are of approximately the same magnitude, namely $q^{(2-\beta)/(2+\beta)}$. \\

A substantial portion of this paper will involve estimating the quantity $A(N,q,c_1,c_2)$, on average over $c_1$ and $c_2$. To this end, we let \[ A_0(q,c_1,c_2) = \vert \{ x,y,z \leqslant q: x^2 - y^2 \equiv c_1 \, (\text{mod } q), \, y^2 - z^2 \equiv c_2 \, (\text{mod } q)\}\vert\]  be the number of solutions to the key congruences, in which the variables $x,y,z$ may range over the entire field $\mathbb{Z}/q\mathbb{Z}$. One might reasonably expect that \[A(M,q,c_1,c_2) \approx (M/q)^3 A_0(q,c_1,c_2)\] as long as $M$ is large enough, and so we introduce the difference 
\begin{equation}
\label{eq: Delta difference}
\Delta(M,q,c_1,c_2) : = \Big\vert A(M,q,c_1,c_2) - \Big(\frac{M}{q}\Big)^3 A_0(q,c_1,c_2)\Big\vert.
\end{equation}
The following key technical lemma will be proved in Section \ref{Section: Proof of Lemma technical estimate}: \begin{Lemma}
\label{Lemma technical estimate}
If $q$ is an odd prime and $M<q$, then
\begin{equation}
\label{eq: technical estimate}
\sum\limits_{c_1,c_2 \leqslant q}
 \Delta(M,q,c_1,c_2)^2 \ll (\log q)^3M^3 + (\log q)^6q^{2}\end{equation} as $q \rightarrow \infty$.\end{Lemma}

\noindent Lemma \ref{Lemma technical estimate} can be used in turn in Section \ref{Section: Proofs of Lemma 2.3} to prove the following lemma on the average size of the error terms $\Delta(N,q,\overline{a}r_1, \overline{a}r_2)$, which is helpful for analysing the relation (\ref{equation sandwiching key}). 

\begin{Lemma}
\label{Lemma bounding key error term}
Let $\beta \in (0, 3/4)$ and let $\eta >0$. Then, if $\eta\max(\beta^{-1}, (3/4 - \beta)^{-1})$ is small enough, the following holds: for almost all $\alpha \in [0,1]$, for all $s_1,t_1,s_2,t_2 \in \mathbb {R}$ for which $s_1 < t_1$ and $s_2 < t_2$, for all $N \in \mathbb{N}$ and for all $L \in \mathbb{R}$ in the range $N^{1 - \beta - \eta}<L<N^{1- \beta + \eta}$, and for all prime $q$ and $a/q$ satisfying \eqref{eq: diophantine approx} and \eqref{equation relationship of N and q}, we have
\begin{equation}
\label{saving in Delta non-zero}
\frac{1}{N}\sum\limits_{\substack{(r_1,r_2) \in S^+ \\ r_1r_2 \neq 0 \\ r_1 + r_2 \neq 0}} \Delta(N,q,\overline{a}r_1, \overline{a}r_2) \ll_{\alpha, \beta, \eta,\mathbf{s},\mathbf{t}} L^2 q^{-\eta},
\end{equation}
\noindent  where $S^{+}$ is as defined in \eqref{eq S plus}.  
\end{Lemma}

At this point it is worth us taking a small diversion from the main proof to discuss the numerology in Lemma \ref{Lemma technical estimate}, and why the bound is close to best-possible. We begin with an easy lemma concerning $A_0(q,c_1,c_2)$ itself. A more sophisticated version of this lemma
was worked out by Kurlberg and Rudnick in \cite[Prop. 4]{KR99}, but to make our exposition self-contained we decided to include a direct proof for the triple correlation case.

\begin{Lemma}\DIFaddbegin \label{lem: approx size of A_not}
\DIFaddend Let $q$ be an odd prime. If $c_1, c_2$ are not both
divisible by $q$, then
\[ A_0(q,c_1,c_2) = \begin{cases}
q + O(\sqrt{q}) & \text{if } c_1 + c_2 \not\equiv 0 \, (\text{mod }q); \\  2q-1 - \Big( \frac{-c_2}{q}\Big) & \text{if } c_1 \equiv 0 \, (\text{mod } q)\\
2q - 1 - \Big( \frac{c_1}{q}\Big) & \text{if } c_2 \equiv 0 \, (\text{mod } q) \\
 2q - 1 - \Big(\frac{c_2}{q}\Big) & \text{if } c_1 + c_2 \equiv 0 \, (\text{mod }q)\end{cases}  \]
where $(\frac{\cdot}{q})$ denotes the Legendre symbol modulo $q$. Further, $A_0(q,0,0)=4q-3$.
\end{Lemma}
 \emph{Remark}: For this paper, it would have been enough to consider only the non-degenerate case of Lemma \ref{lem: approx size of A_not}. However, the arguments of Section \ref{Section: Proof of Lemma technical estimate} are cleaner if we allow ourselves to include the degenerate cases, which is the reason why we do so. 
\begin{proof}
In the trivial case $c_1\equiv c_2 \equiv 0 \,(\text{mod }q)$,
we note that $y=\pm x \,(\text{mod }q)$ and 
$z= \pm y \,(\text{mod }q)$. So, except for when $x=q$, there are $4$ choices for 
$y,z$ given a fixed $x$, thus showing that $A_q(0,0) = 4(q-1) + 1 = 4q-3$ as claimed.

Thus we assume in the following that at least one of $c_1,c_2$ 
is not divisible by $q$. We have that 
\begin{align*}
A_0(q,c_1,c_2) &= \frac{1}{q^2}\sum\limits_{\substack{0 \leqslant j,k \leqslant q-1 \\ 0\leqslant x,y,z \leqslant q-1}}e_q(j(x^2 - y^2 - c_1) + k(y^2 - z^2 - c_2)) 
\end{align*}
The terms when $j = 0$, $k=0$, or $j=k$ contribute
\[ \frac{1}{q^2}\Big(q^2 \sum\limits_{ \substack{0 \leqslant x,y \leqslant q-1 \\ y^2 \equiv x^2 - c_1 \, (\text{mod }q)}} 1  + q^2 \sum\limits_{ \substack{0 \leqslant y,z \leqslant q-1 \\ y^2 \equiv z^2 +c_2 \, (\text{mod }q)}} 1 + q^2 \sum\limits_{ \substack{0 \leqslant x,z \leqslant q-1 \\ x^2 = z^2 + c_1 + c_2}}1 - 2q^3\Big), \] where the final term is used to correct for the overcounting of $(j,k) = (0,0)$. By factorising the ranges of summation using the difference of two squares, this  expression is equal to  \[ q + (q-1)(1_{q\vert c_1}+1_{q \vert c_2}+1_{q\vert(c_1+c_2)}).\]

Recall the Gauss sum evaluation 
\begin{equation}
\label{Gauss sum}
\sum\limits_{ 0 \leqslant x \leqslant q-1} e_q(j x^2) = \Big(\frac{j}{q}\Big) \varepsilon_q \sqrt{q},
\end{equation} when $q \nmid j$ and where $\varepsilon_q = 1$ if $q \equiv 1 \, (\text{mod }4)$ and $\varepsilon_q = i$ if $q \equiv 3 \, (\text{mod }4)$, see \mbox{%DIFAUXCMD
\cite[Thm. 3.3]{Iw04}}. Therefore, the contribution from the remaining frequencies $j,k$ is exactly \[ \frac{\varepsilon_q^3}{q^{1/2}}\sum\limits_{ \substack{1 \leqslant j,k \leqslant q-1 \\ j \neq k}} \Big( \frac{j}{q}\Big) \Big( \frac{k-j}{q}\Big) \Big(\frac{-k}{q}\Big) e_q(-jc_1) e_q(-k c_2). \] By introducing a change of variables $k = lj$, 
this expression is equal to 
\begin{equation}
\label{before alternative Gauss sum}\frac{\varepsilon_q^3}{q^{1/2}}\Big(\frac{-1}{q}\Big) 
\sum\limits_{ 2 \leqslant l \leqslant q-1} 
\Big(\frac{l-1}{q}\Big) \Big( \frac{l}{q}
\Big) \sum\limits_{1 \leqslant j 
\leqslant q-1}\Big(\frac{j}{q}\Big) e_q(j(-c_1 - lc_2)).
\end{equation} 
One can evaluate the inner sum of (\ref{before alternative Gauss sum}) using (\ref{Gauss sum}). Indeed, for an arbitrary integer $m$ and an arbitrary quadratic non-residue $h$ we have
\begin{align*}
\sum\limits_{ 1 \leqslant j \leqslant q-1} \Big(\frac{j}{q}\Big) e_q(jm) &= \frac{1}{2}\Big(\sum\limits_{1 \leqslant n \leqslant q-1} e_q(mn^2) - \sum\limits_{1 \leqslant n \leqslant q-1} e_q(mn^2h)\Big)\\
& = \frac{1}{2}\Big(\sum\limits_{0 \leqslant n \leqslant q-1} e_q(mn^2) - \sum\limits_{0 \leqslant n \leqslant q-1} e_q(mn^2h)\Big) \\
&= \frac{1}{2}\Big(\Big(\frac{m}{q}\Big) - \Big(\frac{hm}{q}\Big)\Big)\varepsilon_q \sqrt{q}\\
& = \Big(\frac{m}{q}\Big) \varepsilon_q \sqrt{q}.
\end{align*}
Plugging this expression into (\ref{before alternative Gauss sum}) we establish that (\ref{before alternative Gauss sum}) is equal to 
\[ \sum\limits_{2 \leqslant l \leqslant q-1}
\Big(\frac{l(l-1)(lc_2 + c_1)}{q}\Big).
\] This is $O(\sqrt{q})$ 
by Hasse (see \mbox{%DIFAUXCMD
\cite[(14.32)]{Iw04}}\hspace{0pt}), 
provided that neither $q \vert c_1$, $q\vert c_2$, nor $q \vert (c_1 + c_2)$. 

In the singular cases, if $q$ divides $c_1$ we end up with \[\Big(\frac{c_2}{q}\Big) \sum\limits_{ 2 \leqslant l \leqslant q-1} \Big( \frac{l-1}{q}\Big),\] which is equal to $-(\frac{-c_2}{q}) $. If $q \vert (c_1 + c_2)$, we end up with 
\[ \Big(\frac{c_2}{q}\Big) \sum\limits_{ 
2 \leqslant l \leqslant q-1} \Big(\frac{l}{q}\Big),\] 
which is equal to $- (\frac{c_2}{q})$. The final case, when $q \vert c_2$, follows easily from the case $q \vert c_1$ after permuting the variables $x,y,z$ in the original expression for $A_0(q,c_1,c_2)$. 
\end{proof}

Therefore $(M/q)^3 A_0(q,c_1,c_2) \asymp M^3 q^{-2}$, and this is the expected size of $A(M,q,c_1,c_2)$. We note, then, that the $M^3$ term in (\ref{eq: technical estimate}) represents `square-root cancellation on average' for the size of $\Delta(M,q,c_1,c_2)$. 

Moreover, Lemma \ref{Lemma technical estimate} is close to best possible, at least for certain ranges of $M$. We would like to 
emphasise that here and throughout, 
$M$ denotes a positive integer. Indeed, if $M  < \lambda q^{2/3}$
and $\lambda$ is a suitably small constant, then 
we have the matching lower bound \begin{equation}
\label{lower bound on variance}\sum\limits_{c_1,c_2 \leqslant q} \Delta(M,q,c_1,c_2)^2 \gg M^3.
\end{equation}

\begin{proof}[Proof of (\ref{lower bound on variance})]
Certainly \[ \sum\limits_{c_1,c_2 \leqslant q} A(M,q,c_1,c_2) = M^3.\] Furthermore, \[ 
\sum\limits_{c_1,c_2 \leqslant q} A(M,q,c_1,c_2)^2 = \sum\limits_{\substack{x,y,z,x^\prime,y^\prime,z^\prime \leqslant M\\ x^2 - (x^\prime)^2 \equiv y^2 - (y^\prime)^2 \equiv z^2 - (z^\prime)^2 \, (\text{mod }q)}} 1.\] By using the divisor bound to control the terms arising when $x^2 - (x^\prime)^2 \neq 0$, this sum is at most $O(q^{o(1)}M^2 q^{-1} + M^3)$, which is certainly at most $O(M^3)$. 

Therefore, by Cauchy-Schwarz 
\[\sum\limits_{c_1,c_2 \leqslant q} 1_{A(M,q,c_1,c_2) \geqslant 1} \geqslant  \Big(\sum\limits_{c_1,c_2 \leqslant q} A(M,q,c_1,c_2) \Big)^2\Big(\sum\limits_{c_1,c_2 \leqslant q} A(M,q,c_1,c_2)^2 \Big)^{-1} \gg M^3.\] If $Mq^{-2/3}$ if sufficiently small then $(M/q)^3A_0(q,c_1,c_2) < 1/2$ for all $c_1,c_2$. Hence \[ \sum\limits_{ c_1,c_2 \leqslant q} 
\Delta(M,q,c_1,c_2)^2 \gg \sum\limits_{c_1,c_2 \leqslant q} 
1_{A(M,q,c_1,c_2) \geqslant 1} \gg M^3\] as claimed. 
\end{proof}

Rudnick--Sarnak--Zaharescu \cite{RSZ01} 
also considered $\Delta(M,q,c_1,c_2)$. 
By Fourier-expanding the cut-off $x,y,z \leqslant M$, they derived
\begin{align*}
\Delta(M,q,c_1,c_2) &\leqslant  \sum\limits_{ \substack{0 \leqslant b_1,b_2,b_3 \leqslant q-1 \\ \mathbf{b} \neq \mathbf{0}}} \prod\limits_{i=1}^3 \vert \widehat{1}_{[M]}(b_i)\vert \Big\vert \sum\limits_{\substack{0 \leqslant x,y,z \leqslant q-1 \\ x^2 - y^2 \equiv c_1 \, (\text{mod } q) \\ y^2 - z^2 \equiv c_2 \, (\text{mod } q)}} e_q(b_1x + b_2y + b_3 z)\Big\vert,
\end{align*} where 
\begin{equation}
\label{fourier transform} \widehat{1_{[M]}}(b) = \frac{1}{q} \sum\limits_{x \leqslant M} e_q(-bx).
\end{equation}
\noindent In expression (9.15) of \cite{RSZ01} they used the Weil bound\footnote{For triple correlations the relevant curve has genus 1, so this is in fact the same Hasse bound as we used in Lemma \ref{lem: approx size of A_not}.}, ending up with a bound of 
\begin{equation}
\label{eq: Weil bound}
\Delta(M,q,c_1,c_2) \ll q^{1/2} (\log q)^3,
\end{equation} in the non-degenerate cases. Comparing this result to Lemma \ref{Lemma technical estimate}, estimate (\ref{eq: Weil bound}) implies
\begin{equation}
\label{weak bound} \sum\limits_{c_1,c_2 \leqslant q} \Delta(M,q,c_1,c_2)^2 \ll q^{3 + o(1)}, 
\end{equation} 
\noindent which is weaker than Lemma \ref{Lemma technical estimate}. We note, therefore, that the proof of Lemma \ref{Lemma technical estimate} must utilise the extra averaging in $c_1$ and $c_2$ in a critical way. 

In the introduction we promised to explain how the threshold $N^{3/5}$ in Theorem \ref{Theorem method of RSZ} arises from the arguments of \cite{RSZ01}, and now seems to be an appropriate moment. Indeed, in order to analyse (\ref{equation sandwiching key}), it would be enough to show that \[\Delta(N,q,\overline{a} r_1,\overline{a}r_2) = o((N/q)^3A_0(q,\overline{a} r_1,\overline{a}r_2)),\] since then one could replace $A(N,q,\overline{a} r_1,\overline{a}r_2)$ with $(N/q)^3A_0(q,\overline{a} r_1,\overline{a}r_2)$ (and then evaluate these latter terms explicitly using Lemma \ref{lem: approx size of A_not}). Using the bound (\ref{eq: Weil bound}), this is only possible when $N^3 q^{-2}>q^{1/2}(\log q)^3$, i.e. provided that $N\geqslant q^{5/6 +o(1)}$. This, one notes, is the same as the threshold from Theorem 4 of \cite{RSZ01} taken with $m=3$ (for triple correlations). Noting that $N \approx q^{\frac{2}{2+\beta}}$, this approach succeeds provided that $2/(2+\beta) > 5/6$, i.e. provided that $\beta < 2/5$. From the definition of $\beta$, this implies that $ L$ must be at least $N^{3/5}$. \\

%Divisor bounds not good enough

%We begin with the crude divisor bound \[ A(N,q,c_1,c_2) \ll \sum\limits_{\substack{\vert k_1\vert ,\vert k_2 \vert \leqslant N^2 \\ k_1 \equiv c_1 \, (\text{mod } q) \\ k_2 \equiv c_2 \, (\text{mod } q)}} \tau(k_1) \tau(k_2) \leqslant N^{4 + o(1)} q^{-2}\] provided $c_1,c_2 \not\equiv 0 \, (\text{mod } q)$. Since $L_i q N^{-\beta} + N^2 q^{-1} < q$, provided $q$ is large enough in terms of $L_i$, we then conclude that \[ F_3^\prime((\alpha n^2)_{n=1}^\infty, N, \beta, (L_1,L_2))  =  N^{-3 + 2 \beta} \sum\limits_{\substack{\vert r_1\vert \leqslant \frac{L_1 q}{N^\beta} \\ \vert r_2\vert \leqslant \frac{L_2 q}{N^\beta} }} A(N,q,\overline{a} r_1,\overline{a} r_2) + O(N^{5 + 2 \beta + o(1)} q^{-4}).\] 

Having finished our diversion on the subject of Lemma \ref{Lemma technical estimate} (whose proof is deferred to Section \ref{Section: Proof of Lemma technical estimate}), let us return to the main argument, namely the proof of Theorem \ref{Theorem main theorem step function constrained}. From now on, we assume $\alpha$ satisfies Lemma \ref{Lemma bounding key error term}. Putting this information into expression (\ref{equation sandwiching key}), we derive
\begin{align}
\label{equation sandwiching key replaced}
N^{2}q^{-3}\sum\limits_{\substack{  (r_1,r_2) \in S^{-} \\ r_1r_2 \neq 0 \\ r_1 + r_2 \neq 0}} A_0(q,\overline{a} r_1,\overline{a} r_2) - O_{\alpha,\beta,\eta, \mathbf{s}, \mathbf{t}}(L^2q^{-\eta}) \leqslant R_3^\prime(\alpha,L,N, g_{\mathbf{s},\mathbf{t}}) \nonumber \\  \leqslant N^{2}q^{-3} \sum\limits_{\substack{  (r_1,r_2) \in S^{+} \\ r_1r_2 \neq 0 \\ r_1 + r_2 \neq 0}}A_0(q,\overline{a} r_1,\overline{a} r_2) + O_{\alpha,\beta,\eta, \mathbf{s}, \mathbf{t}}(L^2q^{-\eta}) .
\end{align}

To estimate the terms in the expression (\ref{equation sandwiching key replaced}), we use Lemma \ref{lem: approx size of A_not}. Indeed
\begin{align}
\label{sum of A0}
N^2 q^{-3} \sum\limits_{\substack{  (r_1,r_2) \in S^{\pm} \\ r_1r_2 \neq 0 
\\ r_1 + r_2 \neq 0}}A_0(q,\overline{a} r_1,\overline{a} r_2) &= N^2 q^{-3} \sum\limits_{\substack{  (r_1,r_2) \in S^{\pm} \\ r_1r_2 \neq 0 \\ r_1 + r_2 \neq 0}} (q + O(q^{1/2})).
\end{align}
The size of $S^{\pm}$ is given by
$$
\Big( (t_1 - s_1) \frac{qL}{N} \pm 2 \frac{N^2}{q^{1-\eta}} + O(1)\Big)\Big( (t_2 - s_2) \frac{qL}{N} \pm 2 \frac{N^2}{q^{1-\eta}} + O(1)\Big).
$$
From relation \eqref{equation relationship of N and q}, one observes that \[N^2 q^{-1 + \eta} = o_{\beta,\eta}\Big(\frac{qL}{N}\Big)\] as $N \rightarrow \infty$, and therefore the size of $S^{\pm}$ is seen to be
\begin{equation}
\label{size of Spm}
\vert S^{\pm}\vert = 
(1+ o_{\beta,\eta, \mathbf{s}, \mathbf{t}}(1))(t_1 - s_1)  (t_2 - s_2) q^2 L^2 N^{-2}
\end{equation} as $N \rightarrow \infty$. The estimate (\ref{size of Spm}) remains after we remove those pairs $(r_1,r_2) \in S^{\pm}$ with $r_1 r_2 = 0$ or $r_1 + r_2 = 0$. 

Thus, returning to (\ref{sum of A0}), we conclude that
\begin{align*}
N^2 q^{-3} \sum\limits_{\substack{  (r_1,r_2) \in S^{\pm} \\ r_1r_2 \neq 0 
\\ r_1 + r_2 \neq 0}}A_0(q,\overline{a} r_1,\overline{a} r_2)
= (1 + o_{\beta, \eta, \mathbf{s}, \mathbf{t}}(1))L^{2} 
(t_1 - s_1)(t_2 - s_2).
\end{align*}
Substituting this estimate into (\ref{equation sandwiching key replaced}), we derive Theorem \ref{Theorem main theorem step function constrained} as required. \qed  \\

What remains is to verify that almost all $\alpha$ admit an approximation $a/q$ of the form required in \eqref{eq: diophantine approx} and \eqref{equation relationship of N and q}, and to prove Lemma \ref{Lemma technical estimate} and Lemma \ref{Lemma bounding key error term}. \\

\section{Approximation with prime denominator}
\label{Section: approximation with prime denominator}
To derive a suitable approximation $a/q$ with prime denominator, we use the following 
quantitative version of (a generalized) 
Khintchine's theorem, due to Harman:

\begin{Theorem}
{\mbox{%DIFAUXCMD
\cite[Thm. 4.2]{Ha98}}\hspace{0pt}%DIFAUXCMD
}
Let $\psi:\mathbb{N}\rightarrow\left(0,1\right)$ be
a non-increasing function such that 
\[\Psi\left(N\right)=\sum_{n\leq N}\psi\left(n\right)
\]is unbounded. For $\mathcal{B}$ an infinite set of integers, let
$S\left(\mathcal{B},\alpha,N\right)$ denote the number of $n\leq N$,
with $n\in\mathcal{B}$, such that $\left\Vert n\alpha\right\Vert <\psi\left(n\right)$.
Then, for almost all $\alpha$, we have 
\begin{equation}
S\left(\mathcal{B},\alpha,N\right)=2\Psi(N,\mathcal{B})+O_{\varepsilon}(\left(\Psi(N)\right)^{\frac{1}{2}}(\log\Psi(N))^{2+\varepsilon})\label{eq: asymptotic formula}
\end{equation}
for each $\varepsilon>0$ where \[\Psi\left(N,\mathcal{B}\right)=\sum_{n\in\mathcal{B}\cap\left[N\right]}\psi\left(n\right).\]
Moreover, the implied constant in (\ref{eq: asymptotic formula}) is uniform in $\alpha$.
\end{Theorem}
From this we deduce the following:
\begin{Lemma}\label{lem: prime moduli}
Let $\eta\in(0,1)$. Then, for almost
all $\alpha$, for all $N \geqslant N_0(\eta)$ there is a prime $q$ satisfying
\begin{equation}
\left\Vert q\alpha\right\Vert <\frac{1}{N^{1-\eta}},\quad\mathrm{and}\quad N\leqslant q\leq2N.\label{eq: prime denominator}
\end{equation}
\end{Lemma}
\begin{proof}
We let $\mathcal{B}$ denote the set of primes, and $\psi\left(n\right)=n^{-1+\eta}$.
Then $\Psi(N) \sim \eta^{-1}N^\eta$ and (by the prime number theorem) $\Psi(N,\mathcal{B}) \sim \eta^{-1}N^{\eta}(\log N)^{-1}$. So, the asymptotic formula (\ref{eq: asymptotic formula}) shows that for almost all $\alpha$ one has
that 
\begin{align*}
S\left(\mathcal{B}, \alpha, N\right) & =\frac{2N^{\eta}\left(1+o\left(1\right)\right)}{\eta\log N}+O(\eta^{-1}N^{\frac{\eta}{2}}(\log N)^{3}),
\end{align*}
\noindent with the implied constant uniform in $\alpha$. From this it immediately follows that if $N$ is sufficiently large in terms of $\eta$ then \[S(\mathcal{B},\alpha,2N) - S(\mathcal{B},\alpha,N) >0,\] as required. \end{proof} Therefore an approximation $a/q$ may be found, with $q$ prime, that satisfies \eqref{eq: diophantine approx} and \eqref{equation relationship of N and q}. \\

\section{Proof of Lemma \ref{Lemma technical estimate}}\DIFaddbegin \label{Section: Proof of Lemma technical estimate}
\DIFaddend 

Expanding the square we have
\[ \sum\limits_{c_1,c_2 \leqslant q} \Delta(N,q,c_1,c_2)^2 = S_1 - 2S_2 + S_3,\] where
\begin{align*}
S_1 &= \sum\limits_{c_1,c_2 \leqslant q} A(M,q,c_1,c_2)^2, \\
S_2 & = \Big(\frac{M}{q}\Big)^3\sum\limits_{c_1,c_2 \leqslant q} A(M,q,c_1,c_2) A_0(q,c_1,c_2), \\
S_3 & = \Big(\frac{M}{q}\Big)^6\sum\limits_{c_1,c_2 \leqslant q} A_0(q,c_1,c_2)^2.
\end{align*}
\noindent We can write each $S_i$ as the number of solutions to certain equations, namely 

\begin{align*}
S_1& = \sum\limits_{\substack{1 \leqslant x,y,z \leqslant M \\ 1 \leqslant x^\prime, y^\prime, z^\prime \leqslant M \\ x^2 - y^2 \equiv  (x^{\prime})^{2} - (y^{\prime})^{2} \, (\text{mod } q) \\ y^2 - z^2 \equiv(y^{\prime})^2 - (z^{\prime})^2 \, (\text{mod } q)}} 1 ,\\
S_2& = \Big(\frac{M}{q}\Big)^3\sum\limits_{\substack{1 \leqslant x,y,z \leqslant M \\ 1\leqslant x^\prime, y^\prime, z^\prime \leqslant q \\ x^2 - y^2 \equiv (x^{\prime})^{2} - (y^{\prime})^{2}\, (\text{mod } q) \\ y^2 - z^2 \equiv(y^{\prime})^2 - (z^{\prime})^2\, (\text{mod } q)}} 1 ,\\
S_3& = \Big(\frac{M}{q}\Big)^6 \sum\limits_{\substack{1 \leqslant x,y,z \leqslant q \\ 1 \leqslant x^\prime, y^\prime, z^\prime \leqslant q \\ x^2 - y^2 \equiv (x^{\prime})^{2} - (y^{\prime})^{2} \, (\text{mod } q)\\ y^2 - z^2 \equiv (y^{\prime})^2 - (z^{\prime})^2 \, (\text{mod } q)}} 1 .\\
\end{align*}

Expanding the cut-offs $1 \leqslant x,y,z, x^\prime, y^\prime, z^\prime \leqslant M$ in terms of additive characters we have 
\[ S_1 = \sum\limits_{ \substack{ 
\mathbf{b} = (b_1,b_2,b_3,b_4,b_5,b_6),\\
0 \leqslant b_1,b_2,b_3,b_4,b_5,b_6 \leqslant q-1 }} 
S(\mathbf{b},q) \prod\limits_{i=1}^6 \widehat{ 1_{[M]}}(b_i),\] 
where \[S(\mathbf{b},q) := \sum\limits_{\substack{x,y,z \leqslant q 
\\ x^\prime,y^\prime,z^\prime \leqslant q \\ x^2 - y^2 
\equiv  (x^{\prime})^{2} - (y^{\prime})^{2} \, (\text{mod } q) 
\\ y^2 - z^2 \equiv(y^{\prime})^2 - (z^{\prime})^2 \, 
(\text{mod } q)}} e_q(\mathbf{b} \cdot 
(x,y,z,x^\prime,y^\prime,z^\prime))\] and $\widehat{1_{[M]}}$ is as in (\ref{fourier transform}). The contribution from the term with 
$\mathbf{b} = \mathbf{0}$ is equal to $S_3$. 
Performing the same expansion on $S_2$, 
we see that the terms arising from 
$\mathbf{b} = \mathbf{0}$ cancel, and we are left with the bound 

\begin{equation}
\label{the first fourier bound}
S_1 - 2S_2 + S_3 \leqslant  \sum\limits_{ \substack{ 0 \leqslant b_1,b_2,b_3,b_4,b_5,b_6 \leqslant q-1 \\ \mathbf{b} \neq \mathbf{0}}}
\Big(\prod\limits_{i=1}^6\vert \widehat{1_{[M]}}(b_i)\vert\Big) \vert S(\mathbf{b},q)\vert .\\
\end{equation}

Our task moves to bounding $\vert S(\mathbf{b},q)\vert$. We have the trivial bound 
\[ \vert S(\mathbf{b},q)\vert \leqslant q^6,\] 
but since $q$ is prime and $\mathbf{b} \neq \mathbf{0}$ we will be able to improve 
on this bound substantially.

Our approach will be elementary. We begin with writing
\begin{align}
\label{equation introducing j and k}
S(\mathbf{b},q) = \frac{1}{q^2} \sum\limits_{ j,k \leqslant q}\sum\limits_{\substack{x,y,z \leqslant q \\ x^\prime,y^\prime,z^\prime \leqslant q }} e_q(&\mathbf{b} \cdot (x,y,z,x^\prime,y^\prime,z^\prime) \nonumber \\&+ j(x^2 - y^2 - (x^\prime)^2 + (y^\prime)^2) + k (y^2 - z^2 - (y^\prime)^2 + (z^\prime)^2).
\end{align}
\noindent As we did when estimating $A_0(q,c_1,c_2)$, let us first consider the contribution from those terms when $j = 0$, $k=0$, or $j=k$. This is 

\begin{align}
\label{j equals zero etc}
&\frac{1}{q^2} \sum\limits_{ k \leqslant q } 
\sum\limits_{\substack{x,y,z \leqslant q \\ 
x^\prime,y^\prime,z^\prime \leqslant q}} 
e_q(\mathbf{b} \cdot (x,y,z,x^\prime,y^\prime,z^\prime) 
+ k(y^2 - z^2 - (y^\prime)^2 + (z^\prime)^2))\nonumber \\
+ &\frac{1}{q^2} \sum\limits_{ j \leqslant q } 
\sum\limits_{ \substack{x,y,z \leqslant q 
\\ x^\prime,y^\prime,z^\prime \leqslant q}} 
e_q(\mathbf{b} \cdot (x,y,z,x^\prime,y^\prime,z^\prime) 
+ j(x^2 - y^2 - (x^\prime)^2 + (y^\prime)^2)) \nonumber\\
+ & \frac{1}{q^2}  \sum\limits_{ j \leqslant q } 
\sum\limits_{\substack{x,y,z \leqslant q \\ 
x^\prime,y^\prime,z^\prime \leqslant q}} 
e_q(\mathbf{b} \cdot (x,y,z,x^\prime,y^\prime,z^\prime) 
+ j(x^2 - z^2 - (x^\prime)^2 + (z^\prime)^2)) \nonumber\\
- & 
\frac{2}{q^2}\sum\limits_{\substack{x,y,z \leqslant q 
\\ x^\prime,y^\prime,z^\prime \leqslant q}} 
e_q(\mathbf{b} \cdot (x,y,z,x^\prime,y^\prime,z^\prime)).
\end{align} 
\noindent 
The first three of these terms devolve into the exponential 
sums involved in the pair correlations of the fractional parts of 
$\alpha n^2$ considered by Heath-Brown in \cite{HB10}. 
Indeed, note that by completing the square in the variables $y, y^\prime, z, z^\prime$ we have 
\begin{equation}
\label{equation first term}
\frac{1}{q^2}\sum\limits_{ k \leqslant q } 
\sum\limits_{\substack{x,y,z \leqslant q 
\\ x^\prime,y^\prime,z^\prime \leqslant q}} 
e_q(\mathbf{b} \cdot (x,y,z,x^\prime,y^\prime,z^\prime) 
+ k(y^2 - z^2 - (y^\prime)^2 + (z^\prime)^2))
\end{equation} is equal to 
\[1_{q \vert b_1} 1_{q \vert b_4}
\sum\limits_{ k \leqslant q-1} \vert G(k)\vert^4 
e_q\Big(\frac{-b_2^2 + b_3^2 + b_5^2 - b_6^2}{4k}\Big),
\] 
where \[ G(k) = \sum\limits_{ x \leqslant q} e_q(k x^2)\] 
is the Gauss sum, as before, and
$1_{q \vert b}$ abbreviates the indicator function of the condition $q\mid b$.
We also use $\frac{1}{k}$ to refer to the multiplicative inverse of $k$ modulo $q$. 
Since $\vert G(k)\vert = q^{1/2}$ by the standard evaluation (\ref{Gauss sum}), 
the term (\ref{equation first term}) is equal to 
\begin{equation}
\begin{cases}
q^3 - q^2 & \text{if } q \vert b_1, q \vert b_4, 
q \vert (-b_2^2 + b_3^2 + b_5^2 - b_6^2)
\\
-q^2 & \text{if } q\vert b_1, q \vert b_4, 
q \nmid (-b_2^2 + b_3^2 + b_5^2 - b_6^2) \\
0 & \text{if } q \nmid b_1 \text{ or } 
q \nmid b_4.
\end{cases}
\end{equation}

We compute that the overall size of (\ref{j equals zero etc}) is
\begin{equation}
\label{j equals 0 etc cases}
\begin{cases}
O(q^3) & \text{if }
q\vert b_1, q \vert b_4, 
q \vert (-b_2^2 + b_3^2 + b_5^2 - b_6^2) \\
O(q^3) & \text{if }
q\vert b_2, q \vert b_5, 
q \vert (-b_3^2 + b_1^2 + b_6^2 - b_4^2) \\
O(q^3) & \text{if }
q\vert b_3, q \vert b_6, q \vert (-b_1^2 + b_2^2 + b_4^2 - b_5^2)\\
O(q^2) & \text{otherwise}.
\end{cases}
\end{equation}

Now consider the contribution to (\ref{equation introducing j and k}) when $j \neq 0$, $k \neq 0$, and $j \neq k$. By completing the square again, this contribution is equal to 

\begin{align}
\label{equation first one with Gauss sums}
\frac{1}{q^2}
\sum\limits_{\substack{j,k \leqslant q-1 \\ j \neq k}} 
\vert G(j)\vert ^2\vert G(k-j)\vert ^2 \vert 
G(-k)\vert^2 
e_{q}\Big(-\frac{b_1^2}{4j} - \frac{b_2^2}{4(k-j)} 
- \frac{b_3^2}{4k}+\frac{b_4^2}{4j} 
+ \frac{b_5^2}{4(k-j)} + \frac{b_6^2}{4k}\Big).
\end{align}
\noindent This is equal to 
\[ q \sum\limits_{\substack{j,k \leqslant q-1 
\\ j \neq k}}e_q\Big( \frac{f(j,k)}{g(j,k)}\Big)
\] 
where 
\[ f(j,k) = j^2 (b_3^2 - b_6^2) 
+ jk(b_1^2 - b_2^2 - b_3^2 - b_4^2 
+ b_5^2 + b_6^2) + k^2 (-b_1^2 + b_4^2)
\] 
and \[ g(j,k) = 4 jk(-j+k).\]

By letting $l = j/k \, (\text{mod } q)$, we can reparametrise this exponential sum as 
\begin{equation}
\label{equation loads of cases}
 q \sum\limits_{\substack{k \leqslant q-1 \\ 
 2 \leqslant l \leqslant q-1}}
 e_q\Big(-\frac{f(l,1)}{k g(l,1)}\Big).
\end{equation} 
This in turn is equal to
\begin{equation}
\label{equation loads of cases 2}
q(q-1)R - q(q-2-R),
\end{equation}
where $R$ is the number of $\ell$ in the range $2 \leqslant \ell \leqslant q-1$ for which $f(\ell,1) = 0 \, (\text{mod } q)$. There are two relevant cases. If $q\mid (b_1^2 - b_4^2)$, 
$q \mid (b_2^2 - b_5^2)$, and $q \mid (b_3^2 - b_6^2)$, then $R = q-2$ and so \eqref{equation loads of cases 2} is exactly equal to $q(q-1)(q-2)$. Otherwise $f(\ell,1)$ is a non-zero polynomial of degree at most two in $\ell$, and so $R = O(1)$. Hence \eqref{equation loads of cases 2} is $O(q^2)$. 

If we combine our knowledge of the sum (\ref{equation loads of cases}) with our knowledge of the terms with $j=0,k=0,j=k$ from (\ref{j equals 0 etc cases}), this yields:
\begin{equation}
\label{combined bound on S(a,p)}
\vert S(\mathbf{b},q)\vert = \begin{cases}
O(q^3) & \text{if }
q\vert b_1, 
q \vert b_4, q \vert (-b_2^2 + b_3^2 + b_5^2 - b_6^2)\\
O(q^3) & \text{if }
q\vert b_2, q \vert b_5, 
q \vert (-b_3^2 + b_1^2 + b_6^2 - b_4^2)\\
O(q^3) & \text{if }
q\vert b_3, q \vert b_6, 
q \vert (-b_1^2 + b_2^2 + b_4^2 - b_5^2)\\
O(q^3) & \text{if } 
q \text{ divides all of } (b_1^2 - b_4^2), (b_2^2 - b_5^2),  \text{ and }
(b_3^2 - b_6^2)\\
O(q^2) & \text{otherwise.}\end{cases}
\end{equation}

Plugging this bound into \eqref{the first fourier bound}, we infer that 
\begin{equation}
\label{eq: S1S2S3}
S_1-2S_2+S_3 \ll q^3  (T_1 + T_2 )
+ q^2 T_3
\end{equation} where (after a straightforward relabelling of the variables) \begin{align*}
T_1 = \sum\limits_{ 
\substack{ -q/2 < b_1,b_2,b_3,b_4,b_5,b_6 <q/2 
\\ q\vert b_1 \pm b_4, q\vert  b_2 \pm b_5, q \vert b_3 \pm b_6 \\ \mathbf{b} \neq \mathbf{0}}}
\Big(\prod\limits_{i=1}^6\vert \widehat{1_{[M]}}(b_i)\vert\Big),
\\
T_2 = \sum\limits_{ 
\substack{ -b/2 <  b_1,b_2,b_3,b_4< q/2 \\ q \mid b_1^2 -b_2^2+b_3^2-b_4^2 \\ \mathbf{b} \neq \mathbf{0}}}
\frac{M^2}{q^2}
\Big(\prod\limits_{i=1}^4\vert \widehat{1_{[M]}}(b_i)\vert\Big),\\
T_3 = \sum\limits_{ 
\substack{ -q/2 < b_1,b_2,b_3,b_4,b_5,b_6 \leqslant q/2 
\\ \mathbf{b} \neq \mathbf{0}}}
\Big(\prod\limits_{i=1}^6\vert \widehat{1_{[M]}}(b_i)\vert\Big).
\end{align*}
\noindent Indeed, since $q$ is prime, if $q \vert b_1^2 - b_4^2$ then $q\vert (b_1 + b_4)$ or $q \vert (b_1 - b_4)$. This is one of the places in which the assumption that $q$ is prime is particularly convenient (the other being the reparametrisation of $j$ and $k$ to get \eqref{equation loads of cases}). \\

For $-q/2 < b < q/2$, recall the standard bound 
\[ \vert \widehat{1_{[M]}}(b)\vert \ll 
\min \Big( \frac{M}{q}, \frac{1}{\vert b\vert}\Big)
\] 
which produces the estimate 
$$
\sum_{-q/2 < b < q/2} \vert \widehat{1_{[M]}}(b)\vert
\ll 
\sum_{ \vert b \vert  < q/M} \frac{M}{q}
+
\sum_{ q/M \leqslant \vert b \vert  < q/2} 
\frac{1}{\vert b\vert} \ll \log q.
$$In the sum defining $T_1$, if $(b_1,b_2,b_3)$ are fixed then there are $O(1)$ possibilities for $(b_4,b_5,b_6)$. Similarly, in $T_2$, once $(b_1,b_2,b_3)$ are fixed then there are at most $2$ possibilities for $b_4$. Hence \[T_1,T_2 \ll \frac{M^3}{q^3} (\log q)^3 \quad \text{and}\quad T_3\ll (\log q)^6,\] since $M<q$. Substituting these bounds into \eqref{eq: S1S2S3}, the proof of Lemma \ref{Lemma technical estimate} is complete. $\qed$\\
\section{Proof of Lemma \ref{Lemma bounding key error term}}\DIFaddbegin \label{Section: Proofs of Lemma 2.3}
\DIFaddend Let $\beta$ and $\eta$ be as in the statement of Lemma \ref{Lemma bounding key error term}. We begin the proof with another auxiliary lemma, which concerns the quantity
 \[ \Delta^*_\beta( q, c_1,c_2) = \max_{M \leqslant q^{\frac{2}{2+\beta}}}\vert \Delta(M,q,c_1,c_2)\vert.\]

\begin{Lemma}
\label{Lemma uniform}
Let $q$ be an odd prime. Then we have \[\sum\limits_{c_1,c_2\leqslant q} \Delta^*_\beta(q,c_1,c_2)^2 
\ll q^{o(1)} q^{\frac{7 - \beta}{2+\beta}}.\]
\end{Lemma}

 \begin{proof}
 
We will deduce this result from Lemma \ref{Lemma technical estimate}. Taking $1 \leqslant K \leqslant M \leqslant q$, we begin by seeking a bound on 
\begin{equation}
\label{equation square differences}
\sum\limits_{c_1,c_2 \leqslant q} (A( M+K, q,c_1,c_2) - A(M,q,c_1,c_2))^2.
\end{equation} Firstly, $A( M+K, q,c_1,c_2) - A(M,q,c_1,c_2)$ is at most
\begin{align*}
 &\vert \{ x \in (M,M+K],\, y,z \leqslant M+K : x^2 - y^2 \equiv c_1 \, (\text{mod } q), \, y^2 - z^2 \equiv c_2 \, (\text{mod } q)\}\vert \\
+&\vert \{ y \in (M,M+K], \, x,z \leqslant M+K: x^2 - y^2 \equiv c_1 \, (\text{mod } q), \, y^2 - z^2 \equiv c_2 \, (\text{mod } q)\}\vert \\
+&\vert \{ z \in (M,M+K], \, x,y \leqslant M+K: x^2 - y^2 \equiv c_1 \, (\text{mod } q), \, y^2 - z^2 \equiv c_2 \, (\text{mod } q)\}\vert. 
\end{align*}
\noindent Let $E_1(q,c_1,c_2)$ refer to the first quantity, $E_2(q,c_1,c_2)$ refer to the second quantity, and $E_3(q,c_1,c_2)$ refer to the third quantity. We have that expression (\ref{equation square differences}) is at most a constant times 
\[ \sum\limits_{c_1,c_2 \leqslant q}(E_1(q,c_1,c_2)^2 + E_2(q,c_1,c_2)^2 + E_3(q,c_1,c_2)^2).\] 
By a change of variables, we may reduce consideration just to $E_2$. Indeed, if $x^2 - y^2 \equiv c_1 \, (\text{mod }q)$ and $ y^2 - z^2 \equiv c_2 \, (\text{mod } q)$ then $z^2-x^2 \equiv -c_1 -c_2 \, (\text{mod }q)$. 
Hence $E_1(q,c_1,c_2) = E_2(q,-c_1-c_2, c_1)$, and so 
\[\sum\limits_{c_1,c_2 \leqslant q}E_1(q, c_1,c_2)^2 
= \sum\limits_{c_1,c_2 \leqslant q} E_2(q, c_1,c_2)^2.
\] A similar argument works for $E_3(q,c_1,c_2)$. 

Now, $\sum_{c_1,c_2 \leqslant q}E_2(q,c_1,c_2)^2$ is equal to 
\begin{equation}
\label{eq: equation in ranges}
\Big 
\vert \Big\{ y_1,y_2 \in (M,M+K],\, x_1,x_2,z_1,z_2 \leqslant M+K : 
\begin{array}{lr}
x_1^2 - y_1^2 \equiv x_2^2 - y_2^2 (\text{mod } q),\\
 y_1^2 - z_1^2 
\equiv y_2^2 - z_2^2 \, 
(\text{mod } q) 
\end{array}
\Big\}
\Big\vert.
\end{equation}
Fixing integers $k_1,k_2$ and integers $y_1,y_2 \in (M,M+K]$, 
we will now bound the number of solutions 
$(x_1,x_2,z_1,z_2)$ to the system of equations
\begin{align}
\label{fiddly equation}
x_1^2 - x_2^2 &= y_1^2 - y_2^2 + k_1 q \nonumber\\
z_2^2 - z_1^2 & = y_2^2 - y_1^2 + k_2 q.
\end{align}
If both $y_1 ^ 2 - y_2^2 + k_1q \neq 0$ and $y_2^2 - y_1^2 + k_2 q \neq 0$ then, using the divisor bound, we get $q^{o(1)}$ solutions. If $y_1 ^ 2 - y_2^2 + k_1q =0$ and $y_2^2 - y_1^2 + k_2 q \neq 0$ we get $O(q^{o(1)}M)$ solutions. If both $y_1 ^ 2 - y_2^2 + k_1q =0$ and $y_2^2 - y_1^2 + k_2 q = 0$ we get $O(M^2)$ solutions. Note that in such a case we must have $k_1 = -k_2$.

Regarding the other conditions on $(y_1,y_2,k_1,k_2)$, we note that the variables $k_1,k_2$ are both restricted to intervals of length 
$O(M^2/q + 1)$. Further, if $k_1 \neq 0$ 
then the total number of solutions of $y_1,y_2,k_1$ 
to $y_1^2 - y_2^2 = k_1 q$, with $y_1,y_2 \in (M, M+K]$, is 
$O(q^{o(1)} (KM/q + 1))$, since $\vert y_1^2 - y_2^2\vert = O(KM)$.
If, however, $k_1=0$ then of course $y_1 = y_2$ so there are $O(K)$ solutions here. 
Thus, summing over $(y_1,y_2,k_1,k_2)$, 
we bound \eqref{eq: equation in ranges} above by 
\[ q^{o(1)}\Big(K^2\Big(\frac{M^2}{q} + 1\Big)^2 
+ M\Big(\frac{M^2}{q} + 1\Big)\Big( \frac{KM}{q} + 1\Big) + KM\Big(\frac{M^2}{q} + 1\Big)
+ M^2\Big(\frac{KM}{q}  + 1\Big)  + K M^2\Big).
\] Since $K \leqslant M \leqslant q$ we may simplify the above and conclude that 
\begin{equation}
\label{eq: bound on difference}
\sum\limits_{c_1,c_2 \leqslant q}(A(M+K,q,c_1,c_2) - A(M,q,c_1,c_2))^2 \ll q^{o(1)} (K^2 M^4 q^{-2} + KM^2).
\end{equation} 
\noindent One may think of this uper bound as being given by the sum of an average contribution and a diagonal contribution. 

From (\ref{eq: bound on difference}), together with previous lemmas, it turns out that one can control 
\[ \sum\limits_{ 1 \leqslant c_1,c_2 \leqslant q} 
\max\limits_{0 \leqslant H \leqslant K} \Delta(M+H, q, c_1,c_2)^2.\] 
Indeed, since \[ \vert(a_1 - a_2)^2 - (b_1 - b_2)^2\vert 
\ll (a_1 - b_1)^2 + (a_2 - b_2)^2\] for all reals $a_1,b_1,a_2,b_2$, 
we have that
\begin{align}\label{eq: difference max Delta}
&\sum\limits_{c_1,c_2 \leqslant q} 
\max\limits_{0 \leqslant H \leqslant K}\Delta(M+H,q,c_1,c_2)^2 
- \sum\limits_{ c_1,c_2 \leqslant q} 
\Delta(M,q,c_1,c_2)^2 
\end{align}
\noindent is as most a constant times
\begin{align}
\label{overful box}
\sum\limits_{c_1,c_2 \leqslant q} 
(A(M+K,q,c_1,c_2) &- A(M,q,c_1,c_2))^2 \nonumber\\ 
&+ \max\limits_{0 \leqslant H \leqslant K} 
\Big(\frac{(M+H)^3-M^3}{q^3}\Big)^2
\sum\limits_{c_1,c_2 \leqslant q} A_0(q,c_1,c_2)^2.
\end{align}
\noindent From expression (\ref{eq: bound on difference}) and Lemma \ref{lem: approx size of A_not}, we conclude that (\ref{overful box}) is at most 
\[ q^{o(1)}(K^2M^4q^{-2} + KM^2)
\] again. Finally, using Lemma \ref{Lemma technical estimate} 
combined with \eqref{eq: difference max Delta}, 
we end up with the bound \begin{equation}
\label{eq: end up with the bound}
\sum\limits_{ c_1,c_2 \leqslant q} \max\limits_{0 \leqslant H \leqslant K}\Delta(M+H,q,c_1,c_2)^2 \ll  q^{o(1)} (M^3 +q^2) + q^{o(1)}K^2M^4 q^{-2}.
\end{equation} \noindent(The $KM^2$ term has been absorbed into the $M^3$ term.)

Before we use (\ref{eq: end up with the bound}) to prove the lemma, we need the bound \begin{equation}
\label{eq: M is 0}
\sum\limits_{c_1,c_2 \leqslant q} \max_{0 \leqslant H \leqslant K} \Delta(H,q,c_1,c_2)^2 \ll q^{o(1)}(K^6 q^{-2} + K^3).
\end{equation}
\noindent This may be proved by an identical analysis to the one above, proceeding with $M = 0$. \\

Now, returning to the original object of the lemma, we may cover \[\{n \in \mathbb{N} : 1 \leqslant n \leqslant q^{\frac{2}{2+\beta}}\}\] by the interval $[1,K]$ together with at most $q^{\frac{2}{2 + \beta}} K^{-1}$ other intervals each of the form $[M,M+K]$ with $K \leqslant M$. Thus, 
\begin{align*}
\sum\limits_{c_1,c_2 \leqslant q} \Delta^*_\beta(q,c_1,c_2)^2 &\leqslant q^{o(1)}(K^6q^{-2} + K^3) + \sum\limits_{l \leqslant q^{\frac{2}{2+\beta}}K^{-1}} \sum\limits_{ c_1,c_2 \leqslant q} \max\limits_{X \in [lK,(l+1) K)} \Delta(X,q,c_1,c_2)^2 
\\
&\leqslant q^{o(1)}(K^6q^{-2} + K^3) +  q^{o(1)}\sum\limits_{l \leqslant q^{\frac{2}{2+\beta}}K^{-1} } (l^3K^3+q^2 + l^4K^6 q^{-2}) 
\\
&\leqslant q^{o(1)}(K^6q^{-2} + K^3) +  q^{o(1)} (q ^{\frac{8}{2+\beta}}K^{-1}
+ q^{\frac{2}{2+\beta} + 2}K^{-1}
 + q^{\frac{10}{2+\beta} - 2}K).
\end{align*}
Since $\beta \leqslant 1$ always, we have 
$q^{\frac{2}{2+\beta} + 2} \leqslant q^{\frac{8}{2 + \beta}}$, which allows us to reduce matters to the bound \[ q^{o(1)}(K^6 q^{-2} + K^3 +  q^{\frac{8}{2+\beta}} K^{-1} + q^{\frac{10}{2+\beta} - 2}K).\]  
Optimising in $K$, we find that the minimum value is achieved when 
$K \asymp q^{\frac{1 + \beta}{2+\beta}}$, 
in which case the third and fourth terms have the same order of magnitude. We conclude that \[ \sum\limits_{c_1,c_2 \leqslant q} \Delta^*_\beta(q,c_1,c_2)^2  \ll q^{o(1)} q^{\frac{7 - \beta}{2+\beta}}\] and the lemma is proved. 
 \end{proof}

\begin{proof}[Proof of Lemma \ref{Lemma bounding key error term}]   
 Let $C$ be a suitably large absolute constant. For all odd prime $q$, and integers $a$ such that $(a,q) = 1$, define \begin{equation}
D_{\beta, \eta}(a,q) := \sum\limits_{ \substack{r_1,r_2 \neq 0 \\ \vert r_1\vert \leqslant q^{\frac{2-\beta}{2+\beta} + C\eta} \\ \vert r_2 \vert\leqslant q^{\frac{2-\beta}{2+\beta} + C\eta}}} \Delta^*_{\beta}(q, \overline{a} r_1,\overline{a}r_2).
\end{equation}
Now, suppose $\alpha \in [0,1]$ and fractions $a/q$ and $N$ satisfy \eqref{eq: diophantine approx} and \eqref{equation relationship of N and q}. We claim that, if $C$ is large enough, and if $q$ is large enough in terms of $\mathbf{s}$, $\mathbf{t}$, and $\eta$, it follows that 
\begin{equation}
\label{inequaltiy removing N}
\sum\limits_{\substack{ (r_1,r_2) \in S^+ \\ r_1r_2 \neq 0 \\ r_1 + r_2 \neq 0}} 
 \Delta(N,q,\overline{a}r_1, \overline{a} r_2) \leqslant D_{\beta, \eta}(a,q).
\end{equation}
Indeed, since $N \leqslant q^{\frac{2}{2+\beta}}$ we have \[ \Delta(N,q, \overline{a} r_1, \overline{a}r_2) \leqslant \Delta^*_\beta(q,\overline{a} r_1, \overline{a}r_2).\] Also, we have both \[\frac{qL}{N} \ll N^{\frac{2+\beta}{2} - \beta + 11\eta} \ll q^{(\frac{2-\beta}{2} +11\eta)\frac{2}{2+\beta} + } \ll q^{\frac{2-\beta}{2+\beta} + C\eta}\] and \[N^2 q^{\eta-1} \ll N^{2 - (1-\eta)(\frac{2+\beta}{2} + 10\eta)} \ll  q^{(\frac{2-\beta}{2} + C\eta)\frac{2}{2+\beta} } \ll q^{\frac{2-\beta}{2+\beta} + C\eta}.\]  Referring to the definition \eqref{eq S plus} of $S^+$, we have thus settled the inequality \eqref{inequaltiy removing N}.

Therefore, to prove Lemma \ref{Lemma bounding key error term} it suffices to show that, for almost all $\alpha \in [0,1]$, for all $(a,q,N)$ satisfying \eqref{eq: diophantine approx} and \eqref{equation relationship of N and q},
\begin{equation}
\label{eq: D(a,q) requirement}
D_{\beta, \eta}(a,q)L^{-2}N^{-1} \ll_{\alpha,\beta,\eta} q^{-\eta}.
\end{equation} 
\noindent In fact, by (\ref{equation relationship of N and q}), it suffices to show that, for almost all $\alpha \in [0,1]$ and for all $(a,q)$ satisfying \eqref{eq: diophantine approx},
\begin{equation}
\label{eq: D(a,q) requirement no N}
D_{\beta, \eta}(a,q) q^{\frac{-6 + 4 \beta}{2+\beta }} \ll_{\alpha,\beta, \eta} q^{-C\eta}.
\end{equation}
\noindent This expression makes no mention of $N$ or $L$, which will be a technical necessity in the Borel--Cantelli argument to come. \\

To show (\ref{eq: D(a,q) requirement no N}), for each $q$ we define
\[ \operatorname{Bad}_{\beta, \eta}(q) = 
\{ a \leqslant q: (a,q) = 1, 
D_{\beta, \eta}(a,q)q^{\frac{-6 + 4 \beta}{2+\beta}}\geqslant q^{-C\eta}\}.
\] It will be enough, then, to show that for almost all $\alpha \in [0,1]$, only finitely many of the fractions $(a,q)$ satisfying (\ref{eq: diophantine approx}) also satisfy $a \in \operatorname{Bad}_{\beta, \eta}(q)$. \\

To bound the size of $\operatorname{Bad}_{\beta, \eta}(q)$, we first note that \[\vert \operatorname{Bad}_{\beta, \eta}(q)\vert \leqslant q^{\frac{-6 + 4 \beta}{2+\beta}+C\eta}\sum\limits_{ \substack{ 1 \leqslant a\leqslant q \\ (a,q) = 1}} D_{\beta, \eta}(a,q). \] Further, we define $f_{\beta, \eta}(q,c_1,c_2)$ to be the number of triples $(a,r_1,r_2)$ such that
\begin{align*}
\overline{a} r_1 &\equiv c_1 \, (\text{mod }q) \\
\overline{a} r_2 &\equiv c_2 \, (\text{mod }q),
\end{align*}
\noindent for which $1 \leqslant a \leqslant q$, $(a,q) = 1$, $r_1r_2 \neq 0$, and $\vert r_1\vert, \vert r_2\vert \leqslant q^{\frac{2-\beta}{2+\beta} + C\eta}$. We then have 
\begin{align}
\label{returning to}
& \vert \operatorname{Bad}_{\beta, \eta}(q)\vert \leqslant q^{\frac{-6 + 4 \beta}{2+\beta } + C\eta} \sum\limits_{0 \leqslant c_1,c_2 \leqslant q-1} f_{\beta, \eta}(q,c_1,c_2) \Delta^*_\beta(q,c_1,c_2).
\end{align}

To estimate (\ref{returning to}), we need the following simple lemma about the size of $f_{\beta, \eta}(q,c_1,c_2)$:
\begin{Lemma}
\label{Lemma bound on f squared}

We have \[\sum\limits_{c_1,c_2 \leqslant q} f_{\beta, \eta}(q,c_1,c_2)^2 \ll q^{\frac{4 - 2 \beta}{2+\beta} + 2C\eta}(q + q^{\frac{4 - 2 \beta}{2+\beta} + 2C\eta })q^{o(1)}.\]
\end{Lemma}
\begin{proof}
We have that $\sum_{c_1,c_2 \leqslant q} f_{\beta, \eta}(q,c_1,c_2)^2$ is equal to the number of pairs of triples $(a,r_1,r_2)$, $(b,s_1,s_2)$ such that
\begin{align}
\label{full pair of equations}
\overline{a} r_1 &\equiv \overline{b} s_1 \, (\text{mod }q) \nonumber \\
\overline{a} r_2 &\equiv \overline {b} s_2 \, (\text{mod }q),
\end{align} with $1\leqslant a,b \leqslant q$, $(a,q) = 1$, $(b,q) = 1$, $r_1r_2s_1s_2 \neq 0$ and $\vert r_1\vert, \vert r_2 \vert , \vert s_1\vert, \vert s_2\vert \leqslant q^{\frac{2-\beta}{2+\beta} + C\eta}$. By multiplying the first equation by $br_2$ and the second equation by $br_1$, one sees that for every such solution we must also have 
\begin{equation}
\label{simpler equation}
 s_1 r_2 \equiv s_2 r_1 \, (\text{mod } q).
 \end{equation}
  The mapping between the solutions to (\ref{full pair of equations}) and (\ref{simpler equation}) is at most $q$-to-$1$ because if the variables $a,r_1,r_2,s_1,s_2$ are fixed then $b$ is uniquely determined in (\ref{full pair of equations}). 
  
Solutions to (\ref{simpler equation}) are given by solutions 
$(r_1,r_2,s_1,s_2,k)$ to \[ s_1r_2 - s_2r_1 = kq,\] with $k$ 
in the range $0 \leqslant \vert k\vert 
\ll q^{\frac{2(2-\beta)}{2+\beta} - 1 + 2C\eta}$. 
By the divisor bound, after $k$, $s_1$ and $r_2$ are fixed there at most $q^{o(1)}$ valid choices of $s_2$ and $r_1$, so the total number of solutions to (\ref{simpler equation}) is at most 
\[ q^{\frac{4 - 2\beta}{2+\beta} + 2 C\eta}(1 + q^{\frac{4-2\beta}{2+\beta} - 1 + 2C\eta}).\] Multiplying by $q$ to get the number of solutions to (\ref{full pair of equations}), we prove the lemma. 
\end{proof}
\noindent (The reader may wish to note that if $\beta >2/3$ then the only valid value of $k$ in the above is $k = 0$, which simplifies the remainder of the analysis for these cases.)\\

Returning to (\ref{returning to}), we have
\begin{align}
\label{the critical inequalities and CS}
\vert \operatorname{Bad}_{\beta, \eta}(q)\vert &\leqslant q^{\frac{-6 + 4 \beta}{2+\beta } + C\eta}\cdot \Big(\sum\limits_{  c_1,c_2 \leqslant q} f_{\eta, \beta}(q,c_1,c_2)^2 \Big)^{1/2}\cdot \Big( \sum\limits_{c_1,c_2 \leqslant q} \Delta^*_\beta(q,c_1,c_2)^2 \Big)^{1/2}\nonumber  \\
& \ll q^{\frac{-6 + 4 \beta}{2+\beta }} \cdot (q^{\frac{1}{2} + \frac{2-\beta}{2+\beta}} + q^{\frac{4 - 2\beta}{2+\beta}}) \cdot q^{\frac{7 - \beta}{2(2+\beta)}}\cdot  q^{C\eta + o(1)} \nonumber\\
&\ll (q^{\frac{1+6\beta}{2(2+\beta)}} + q^{\frac{3 + 3\beta}{2(2+\beta)}}) q^{C\eta + o(1)}.
\end{align}
Here we have used Lemma \ref{Lemma uniform} and Lemma \ref{Lemma bound on f squared} to go from the first line to the second line.

Now, recall that $\beta < 3/4$, which implies that \[ \frac{ 1 + 6\beta}{2(2 + \beta)} < 1.\] The other term is less severe, and in fact $\beta < 1$ implies \[ \frac{3 + 3\beta}{2(2+ \beta)} < 1.\] Since $\eta$ is small enough, and $C$ is absolute, we conclude that 
\begin{equation}
\label{eq: size of bad}
\vert \operatorname{Bad}(q,\beta,\eta) \vert \ll_{\beta,\eta} q^{1 - 2\eta}.
\end{equation}

Now we can finally complete the proof of Lemma \ref{Lemma bounding key error term} by using the first Borel--Cantelli lemma. Indeed, pick $\alpha$ uniformly at random in $[0,1]$ and for each prime $q \geqslant 3$ let $E_{q,\beta, \eta}$ be the event that there exists an $a$ with $1 \leqslant a \leqslant q$, $(a,q) = 1$, \[\Big\vert \alpha - \frac{a}{q} \Big\vert < \frac{1}{q^{2 - \eta}},\] and $a \in \operatorname{Bad}_{\beta, \eta}(q)$. Then \[\mathbb{P}(E_{q, \beta, \eta}) \leqslant  \sum\limits_{ a \in \operatorname{Bad}_{\beta, \eta}(q)} \mu\Big( \Big(\frac{a}{q} - \frac{1}{q^{2-\eta}}, \frac{a}{q} + \frac{1}{q^{2 - \eta}}\Big)\Big) \ll_{\beta, \eta} q^{-1 -\eta}.\] 

Then $\sum_{q \geqslant 3} \mathbb{P}(E_{q, \beta, \eta}) < \infty$, and so with probability $1$ only finitely many of the events $E_{q,\beta, \eta}$ occur. Thus, by our long chain of reductions, Lemma \ref{Lemma bounding key error term} follows. 
\end{proof}
\vspace{3mm}
\section{Concluding remarks}
\label{Section Concluding remarks}
The proof of all of our main theorems is now complete. However, before concluding the paper, it is certainly worth us discussing whether $\beta < 3/4$ represents a natural limit of our approach. 

The chain of inequalities (\ref{the critical inequalities and CS}) is the critical moment of the entire proof, and this particular application of Cauchy--Schwarz is the main source of our loss in the range of $\beta$. Suppose that instead we had used the bound 
\begin{equation}
\label{eq: instead we used}
\vert \operatorname{Bad}_{\beta, \eta}(q)\vert \leqslant q^{\frac{-6 + 4 \beta}{2+\beta } + C\eta}\Big( \sum\limits_{ c_1,c_2 \leqslant q} f_{\beta, \eta}(q,c_1,c_2)\Big)^{1/2} \Big(\sum\limits_{c_1,c_2 \leqslant q} f_{\beta, \eta}(q,c_1,c_2) \Delta^*_\beta(q,c_1,c_2)^2 \Big)^{1/2}.
\end{equation} For simplicity of exposition here, we will assume that $\beta > 2/3$, that $C = 0$, and we will ignore all $q^{o(1)}$ terms. It is then easy to see that \[ \sum\limits_{c_1,c_2 \leqslant q} f_{\beta, \eta}(q,c_1,c_2) \approx q^{1 + \frac{2(2-\beta)}{2+\beta} }.\] Combining this bound with Lemma \ref{Lemma bound on f squared} one may conclude that  $f_{\beta, \eta}(q,c_1,c_2) \approx 1$ for $q^{1 + \frac{2(2-\beta)}{2+\beta}}$ pairs $(c_1,c_2)$, and is $0$ otherwise. So, given what we know from Lemma \ref{Lemma uniform} about the value of $\Delta^*_\beta(q,c_1,c_2)^2$ averaged over all pairs $c_1$ and $c_2$, it is not utterly unreasonable to hope that one could prove 
\begin{equation}
\label{eq: include the weight in this way}
\sum\limits_{c_1,c_2 \leqslant q} f_{\beta, \eta}(q,c_1,c_2) \Delta^*_\beta(q,c_1,c_2)^2 \ll_{\beta, \eta}  q^{\frac{7 - \beta}{2+\beta}} \cdot q^{1 + \frac{2(2-\beta)}{2+\beta}} \cdot q^{-2}= q^{\frac{9 - 4\beta}{2 + \beta}},
\end{equation} provided that the weight of $\Delta^*_\beta(q,c_1,c_2)^2$ does not concentrate on the support of $f$. 

Putting this bound into (\ref{eq: instead we used}) one would then get
\[\vert \operatorname{Bad}_{\beta,\eta}(q)\vert \ll_{\beta, \eta} q^{\frac{3 + 3 \beta}{2(2+\beta)}},\] i.e. only the second term from (\ref{the critical inequalities and CS}) would occur. As we have already remarked, we would then derive 
\[ \vert \operatorname{Bad}_{\beta, \eta}(q)\vert \ll_{\beta, \eta} q^{1- 2\eta},\] 
provided $\beta < 1$ and $\eta$ is small 
enough. This estimate would 
expand the range of Theorem \ref{Theorem TW} all the way to $L > N^\varepsilon$. Unfortunately, we have not been 
able to prove a version of 
Lemma \ref{Lemma technical estimate} which includes the weight 
$f_{\beta, \eta}(q,c_1,c_2)$ in the manner of expression 
(\ref{eq: include the weight in this way}).

One also recalls that in our application of Borel--Cantelli we did not need to bound $\vert \operatorname{Bad}_{\beta, \eta}(q)\vert$ uniformly for all $q$. One would be satisfied with \[\sum\limits_{q \text{ prime }}\frac{\vert \operatorname{Bad}_{\beta, \eta}(q)\vert}{q^2} < \infty.\] Thoughts move towards expressing the relevant exponential sums as an average of Kloosterman-type sums over the modulus $q$, which might be another route for future research. \\

Our final remark is that if $L \rightarrow \infty$ and $N/L \rightarrow \infty$ as $N \rightarrow \infty$ then, in the random model (\ref{random model}), the asymptotics are governed by the Central Limit Theorem. One can derive \[\frac{Z_{L,N} - L}{\sqrt{L}}\xrightarrow{dist} N(0,1)\] as $N\rightarrow \infty$. Theorem \ref{Theorem TW} could then be considered as a first step towards showing that for almost all $\alpha$ the skewness of $W_{\alpha,L,N}$ satisfies $\mathbb{E} ((W_{\alpha,L,N} - L)/\sqrt{L})^3\rightarrow 0$ as $N \rightarrow \infty$, with $L$ in a certain range. However, to show this asymptotic one would need to be able to extract the lower degree main-term from $\mathbb{E} W_{\alpha,L,N}^3$ (which is $3L^2$, as in Section \ref{Section reduction to correlation functions}) and then subsequently show that the error term in Theorem \ref{Theorem main theorem step function} is in fact $o(L^{3/2})$, rather than merely $o(L^3)$.  \\

\bibliographystyle{plain}
\appendix 

\section{Pair Correlations of the dilated squares at scale $N^{-\beta}$}\label{Appendix A}
In this section, we briefly indicate how one can deduce the following fact from the methods of the literature.  
\begin{Theorem}\label{thm: pair corr at scale < 1}
Let $\varepsilon \in(0,1/4)$. Then for almost all $\alpha \in [0,1]$, for all 
$ 1\leqslant L \leqslant N^{1- \varepsilon}$ and for all 
$(\log N)^{-1} \leqslant s \leqslant \log N$ we have
$$ 
R_2(\alpha ,L, N, 1_{[-s,s]}) =  
2Ls (1 + O_{\varepsilon,\alpha}(N^{-\varepsilon/13})).
$$
\end{Theorem}
\noindent Note that this result immediately implies the estimate (\ref{pair correlation estimate that we need}), by approximating the function $f$ in (\ref{pair correlation estimate that we need}) with a suitable step function. 

We have made no attempt to obtain the best possible error term in Theorem \ref{thm: pair corr at scale < 1}, nor the largest admissible ranges for $s$ and for $L$. One will observe from the proof that rather better bounds would certainly follow if one assumed at the outset that $L$ were a slowly varying function of $N$.

A version of Theorem \ref{thm: pair corr at scale < 1} follows from arguments of Aistleitner, Larcher, 
and Lewko \cite{ALL17} as well as from arguments of Rudnick \cite{Ru18}
and Rudnick--Sarnak \cite{RS98}, by changing the relevant parameters. 
If one wanted an explicit characterisation of the set of suitable 
$\alpha$ in terms of properties of its rational approximations 
then one could also adapt the (much more involved) methods of 
Heath-Brown \cite{HB10}. In particular, 
the material of the present section is in no way novel.
However, we decided to add some explanations
on how to deduce Theorem \ref{thm: pair corr at scale < 1},
partly in order to make the exposition of our previous arguments
complete and self-contained, and partly in order to describe explicitly a suitable `sandwiching argument' for this result (expanding upon the description in \cite{ALL17}). \\

We begin with the following auxiliary lemma 
(where again no attempt was made to obtain 
the best possible error term):

\begin{Lemma}
\label{Lemma m lemma}
For each $m \in \mathbb{N}$, let $N_{m}=m^4$. 
Letting $\varepsilon\in (0,1/4)$, 
for each $i$ in the range 
$-m^{\varepsilon/3}/10 \leqslant i \leqslant m^{\varepsilon/3}$ 
let $\beta_{m,i} = i/m^{\varepsilon/3}$ and $L_{m,i} = N_m^{\beta_{m,i}}$. Let $s_m$ be a real quantity that satisfies 
$ m^{-\varepsilon/10}<s_m < m^{\varepsilon/10}$ 
for large enough $m$. Then, for almost all $\alpha \in [0,1]$, for all $m$ and for all $i$ in the range $-m^{\varepsilon/3}/10 \leqslant i \leqslant m^{\varepsilon/3}$ we have \begin{equation}
\label{subsequence poisson}
R_2(\alpha,L_{m,i},N_m, 1_{[-s_m,s_m]}) 
= 2s_mL_{m,i} + O_{\alpha,\varepsilon}\Big( \frac{L_{m,i}}{ m^{\frac{13}{15} - \frac{\varepsilon}{3}}}\Big).
\end{equation}
\end{Lemma}

\begin{proof}[Proof of Lemma \ref{Lemma m lemma}]

Let $I=\left(\gamma,\delta\right)$ be an arc on the torus 
$\mathbb{R}/ \mathbb{Z}$
such that $0<\gamma-\delta<1$. Let $J\geqslant 1$ be an integer. To proceed
we introduce trigonometric polynomials $S_{J}^{\pm}\left(x\right)$, of
degree $J$, which approximate the indicator function $\chi_{I}$
from above and below. Selberg, and also Vaaler, constructed
such polynomials, cf. Montgomery \cite[p. 5--6]{M94}. Indeed, there exists
\[
S_{J}^{\pm}\left(x\right)=
\sum_{\left|j\right|\leqslant J}s_{J}^{\pm}\left(j\right)e(jx)
\]
satisfying
\[
S_{J}^{-}\left(x\right)
\leqslant
\chi_{I}\left(x\right)
\leqslant S_{J}^{+}\left(x\right)
\qquad(x\in\mathbb{R}/\mathbb{Z})
\]
such that 
\[
s_{J}^{\pm}\left(0\right)=\delta - \gamma \pm\frac{1}{J+1},
\qquad\left|s_{J}^{\pm}\left(j\right)\right|
\leqslant \frac{1}{J+1}
+\min\left(\delta-\gamma,\frac{1}{\pi\left|j\right|}\right)
\qquad\left(0<\left|j\right|\leqslant J\right).
\]
For our purposes, given $N_m$ and some growth function $w(N_m) \geqslant 1$, to be specified later, we specify 
$J_{m,i}= \lfloor N_mw(N_m)/L_{m,i} \rfloor $,
$\gamma_{m,i}=-s_mL_{m,i}/N_m$, and $\delta_{m,i}=s_mL_{m,i}/N_m$. Note that the definitions of $s_m$ and $L_{m,i}$ in the statement of the theorem imply that these choices yield a valid arc on the torus. Then, defining\[
R_{2}^{\pm}(\alpha,L_{m,i},N_m,1_{[-s_m,s_m]})
:=\frac{1}{N_m}\sum_{x\neq y\leqslant N_m}S_{J_{m,i}}^{\pm}(\alpha(x^{2}-y^{2})),
\]
we can control $R_{2}$ from above and below via
\begin{equation}
\label{above and below}
R_{2}^{-}(\alpha,L_{m,i},N_m,1_{[-s_m,s_m]}) 
\leqslant R_{2}(\alpha,L_{m,i},N_m,1_{[-s_m,s_m]})
\leqslant R_{2}^{+}(\alpha,L_{m,i},N_m,1_{[-s_m,s_m]}).
\end{equation}

Let \[ E^{\pm}(L_{m,i},N_m,s_m) := \int\limits_{0}^1 R_2^{\pm}(\alpha,L_{m,i},N_m,1_{[-s_m,s_m]}) \, d\alpha\] be the expected value of $R_2^{\pm}(\alpha,L_{m,i},N_m,1_{[-s_m,s_m]})$. It is easy to see that 
\begin{equation}
\label{expectation estimate}
E^{\pm}(L_{m,i},N_m,s_m) = 2s_mL_{m,i}+ O(L_{m,i}/w(N_m)) + O(s_mL_{m,i}/N_m).
\end{equation} Furthermore, by using orthogonality, we also see that the variance \[
\int_{0}^{1} (R_{2}^{\pm}(\alpha,L_{m,i},N_m,1_{[-s_m,s_m]}) - E^{\pm}(L_{m,i},N_m,s_m))^2 d\alpha\] of $R_2^{\pm}$ is at most \[\leqslant \frac{1}{N_m^2}\sum_{\substack{x_1\neq y_1 \leqslant N_m\\ x_2\neq y_2 \leqslant N_m}}\,
\sum_{\substack{0< \vert j_1\vert,\vert j_2\vert \leqslant J_{m,i}\\ 
j_1(x_1^2 - y_1^2) + j_2 (x_2^2 -y_2^2) = 0}} 
\vert s_{J_{m,i}}^{\pm}(j_1) s_{J_{m,i}}^{\pm} (j_2) \vert.\] (Note that there are no contributions from terms in which $j_1 = 0$ and $j_2 \neq 0$, since the condition $j_2(x_2^2 - y_2^2) = 0$ cannot be satisfied.)\\

Since
$\vert s_{J_{m,i}}^{\pm}(j) \vert \leqslant (2s_m + 1) L_{m,i}/N_m$ 
for all $j$ in the range $0<\vert j\vert  \leqslant J_{m,i}$, we can bound
the variance above by $O((s_m + 1)^2)$ times
\begin{equation}\label{eq: variance for pair corr at scale < 1}
\frac{L_{m,i}^2}{N_m^{4}} 
\Bigg \vert \Bigg\{(j_1,x_1,y_1,j_2,x_2,y_2) \in \mathbb{Z}^6 : 
j_1 (x_1^2 -y_1^2 ) = j_2 (x_2^2 -y_2^2 ),
\begin{array}{lr} 
1\leqslant x_k\neq y_k \leqslant N_m,\\
0 < \vert j_k \vert \leqslant \frac{N_m w(N_m)}{L_{m,i}},\\ 
k=1,2
\end{array}
\Bigg\}
\Bigg\vert.
\end{equation}
\noindent Let us fix the first three variables above, 
that is $j_1,x_1,y_1$. Then 
by writing $x_2^2-y_2^2= d_1 d_2$, with $d_1 = x_2-y_2$ and $d_2 = x_2+y_2$, 
we deduce that $d_k \mid j_1 (x_1^2 -y_1^2 ) $ for $k = 1,2$. 
By the divisor bound, that there are at most
$N_m^{o(1)}$ many possibilities for $d_1,d_2$ (provided that $w(N_m) \leqslant N_m^{O(1)}$). Moreover,
any choice of $d_1,d_2$ uniquely determines
the variables $x_1,y_1$ via $ x_2 = (d_1 + d_2)/2 $
and $ y_2 = (d_2- d_1)/2 $. Further, we note that 
$j_2$ is determined up to $\ll N_m^{o(1)}$ many choices. 
The upshot is that given one of the $O( N_m^3 w(N_m)/L_{m,i})$ 
many admissible choices for $j_1,x_1,y_1$, the second block 
of variables $j_2,x_2,y_2$ is determined up to 
$O(N_m^{o(1)})$ many possibilities. 
Therefore the variance of $R_2^\pm$ is at most
$O((s_m+1)^2L_{m,i} w(N_m) N_m^{-1 + o(1)})$.\\

For the ease of exposition, we let $\kappa_{m,i}^{\pm}=\vert 
R_{2}^{\pm}(\alpha,L_{m,i},N_{m},1_{[-s_m,s_m]})
- E^{\pm}(L_{m,i},N_m,s_m)\vert$.
We infer, by Chebychev's inequality, that 
\[
\mathbb{P}
\left(\alpha\in\left[0,1\right]:
\kappa_{m,i}^{\pm} \geqslant \frac{L_{m,i}}{w(N_m)} 
\right) \leqslant
\frac{(s_m+1)^2w(N_m)^3 N_m^{o(1)}}{N_mL_{m,i}} \leqslant \frac{(s_m+1)^2w(N_m)^3 N_m^{o(1)}}{N_m^{9/10}}.
\] 
Choose the growth function $w(N_m)$ to be
\[
w(N_m) = \Big(\frac{N_m^{9/10}}{m^{1 + \varepsilon}} \Big)^{1/3}.
\] Therefore, since $s_m < m^{\varepsilon/10}$, we conclude that
\begin{align*}
\mathbb{P}\left(\alpha\in\left[0,1\right]:
\kappa_{m,i}^{\pm}\geqslant
\frac{L_{m,i}}{w(N_m)} \right) 
 \ll_{\varepsilon} \frac{1}{m^{1 + \frac{\varepsilon}{2}}}.
\end{align*} Summing over $i$ in the range
$-m^{\varepsilon/3}/10 \leqslant i \leqslant m^{\varepsilon/3}$, with the union bound we get 
\begin{align*}
\mathbb{P}\Big(\alpha\in\left[0,1\right]:
\exists \, i \in [-m^{\varepsilon/3}/10, m^{\varepsilon/3}]
\text{ s.t. }
\kappa_{m,i}^{\pm} \geqslant
\frac{L_{m,i}}{w(N_m)} \Big)
\ll_\varepsilon \frac{1}{m^{1 + \frac{\varepsilon}{6}}}.
\end{align*}

Recalling our estimate (\ref{expectation estimate}), 
the first Borel--Cantelli lemma implies that,
for almost every $\alpha\in\left[0,1\right]$,
the following relation holds for all $m\geqslant 1$ and all admissible $i \in [-m^{\varepsilon/3}/10, m^{\varepsilon/3}]$:
\[
\frac{1}{N_{m}}\sum_{x\neq y\leqslant N_{m}}
R_2^{\pm}(\alpha,L_{m,i},N_m, 1_{[-s_m,s_m]}) = 
2s_mL_{m,i} + O_{\alpha, \varepsilon}\Big( \frac{L_{m,i}}{w(N_m)}\Big) + O_{\alpha}\Big(\frac{s_m}{N_m}\Big).
\]
From (\ref{above and below}), and substituting in the explicit growth function $w(N_m)$, the lemma follows. 
\end{proof}

\begin{proof}[Proof of Theorem \ref{thm: pair corr at scale < 1}]
We begin with the trivial observation that, by combining $s$ and $L$ into a single parameter, it is enough to show that 
for almost all 
$\alpha \in [0,1]$, for all $N$ and for all $L$ in the range 
$ N^{-1/11}\leqslant L \leqslant N^{1 - \varepsilon/2}$,
\begin{equation}\label{eq: comined parameters}
R_2(\alpha,L,N,1_{[-1,1]}) = 2L (1+ O_{\alpha,\varepsilon}(N^{-\frac{\varepsilon}{13}})).
\end{equation}
Now, for each $N$, choose $m$ such that $N_{m}\leq N<N_{m+1}$, 
where $N_m = m^4$ as in Lemma \ref{Lemma m lemma}. 
We put $\theta_{m}=N_{m+1}/N_{m}$. Then, for any $L$,
\[
R_{2}(\alpha,L,N,1_{[-1,1]})
\leqslant \frac{N_{m+1}}{N}
R_{2}\Big(\alpha,L,N_{m+1},1_{\frac{N_{m+1}}{N}[-1,1]}\Big)
\leqslant \theta_{m}R_{2}(\alpha,L,N_{m+1},1_{\theta_{m}[-1,1]}),
\]
and similarly
\[
R_{2}(\alpha,L,N,1_{[-1,1]})
\geqslant \frac{N_{m}}{N}
R_{2}\Big(\alpha,L,N_{m},1_{\frac{N_{m}}{N}[-1,1]}\Big)
\geqslant \theta_{m}^{-1}R_{2}(\alpha,L,N_{m},1_{\theta_{m}^{-1}[-1,1]}).
\] 

For each  
$ N^{-1/11}\leqslant L \leqslant N^{1 - \varepsilon/2}$, 
there exists an $i$ in the range 
$ - m^{\varepsilon/3}/10 \leqslant i\leqslant m^{\varepsilon/3}$ such that 
$L_{m,i} \leqslant L \leqslant L_{m,i+1}$, where $L_{m,i}$ is as in Lemma \ref{Lemma m lemma}. 
Now we record that the upper and lower bounds above satisfy
\begin{align*}
& R_{2}(\alpha,L,N_{m},1_{\theta_{m}^{-1}[-1,1]}) \geqslant
R_{2}(\alpha,L_{m,i},N_{m},1_{\theta_{m}^{-1}[-1,1]}),\\
& R_{2}(\alpha,L,N_{m+1},1_{\theta_{m}[-1,1]}) \leqslant 
R_{2}(\alpha,L_{m,i+1},N_{m+1},1_{\theta_{m}[-1,1]}).
\end{align*}
Moreover, Lemma \ref{Lemma m lemma} 
implies that there is a set $\Omega_{\varepsilon} \subset [0,1]$, 
with full measure, such that, if $\alpha \in \Omega_{\varepsilon}$
and $\varepsilon \in (0,1/4]$, then for all $m\geqslant1$ and $i$ in the range
$ - m^{\varepsilon/3}/10 \leqslant i \leqslant m^{\varepsilon/3}$, 
\[ R_2(\alpha,L_{m,i},N_m, 1_{\theta_m[-1,1]}) 
=  2\theta_m L_{m,i}(1+ O_{\alpha,\varepsilon}
(m^{-\frac{13}{15} + \frac{\varepsilon}{3}})).\] 

By using that
$ L_{m,i+1}/L_{m,i} = 1 + O_{\varepsilon}(N^{-\varepsilon/13})$ and
also that
$ \theta_{m}= 1+ O(m^{-1})$, we infer that
\begin{align*}
R_{2}(\alpha,L,N,1_{[-1,1]})& \leqslant 2L(1+ O_{\alpha,\varepsilon}
(N^{-\frac{1}{4}}))
(1+ O_{\alpha,\varepsilon}(N^{-\frac{\varepsilon}{13}}))(1 + O_{\alpha,\varepsilon}(m^{-\frac{13}{15} + \frac{\varepsilon}{3}})) \\
& \leqslant 2L(1 + O_{\alpha,\varepsilon}(N^{-\frac{\varepsilon}{13}})).
\end{align*}
Similarly, we conclude that 
$$
R_{2}(\alpha,L,N,1_{[-1,1]}) \geqslant
2 L(1+ O_{\alpha,\varepsilon}(N^{-\frac{\varepsilon}{13}})).
$$
Combining these two estimates shows \eqref{eq: comined parameters},
thus completing the proof of Theorem \ref{thm: pair corr at scale < 1}. 
\end{proof}

\section{Discrepancy and $k$-point correlation functions at scale $N^{-\beta}$}\label{Appendix B}
The purpose of the present section is to record a few
simple observations concerning the relationship between
discrepancy and $k$-point correlation functions. 

\begin{Definition}
\label{Definition general correlation functions}
Let $(x_n)_{n=1}^{\infty}$ be a sequence of points in $[0,1)$. Let $k \geqslant 2$ be a natural number, and let $g: \mathbb{R}^{k-1} \longrightarrow [0,1]$ be a compactly supported function. Then the $k^{th}$ correlation function $R_k((x_n)_{n=1}^\infty, L,N,g)$ is defined to be \[R_k((x_n)_{n=1}^\infty, L,N,g): = \frac{1}{N} \sum\limits_{\substack{n_1,\dots,n_k \leqslant N \\ \text{distinct}}} 
g\Big(\frac{N}{L}\{x_1 - x_2\}_{\mathrm{sgn}},
\frac{N}{L}\{x_2 - x_3\}_{\mathrm{sgn}}, \dots, 
\frac{N}{L} \{ x_{k-1} - x_k \}_{\mathrm{sgn}}\Big),\] where \[ \{\cdot\}_{\mathrm{sgn}}:\mathbb{R} \longrightarrow (-1/2,1/2]\] denote the signed distance to the nearest integer. 
\end{Definition}
\noindent The main point we are conveying here is that, as expected, the correlations are controlled on the scales
in which the discrepancy allows us to count points asymptotically. 
\begin{Lemma}\label{lem: discrepancy}
(a) Let $D_N$ denote the discrepancy of the sequence $(x_n)_{n=1}^\infty$
in $[0,1)$, and suppose
$\sup \{ g>0: D_N \ll_{g} N^{-g} \, \text{for all } N\} = \gamma > 0$. Let $k \geqslant 2$ be a natural number, 
and let $Y$ be a uniformly distributed random variable modulo $1$. 
Let $\varepsilon > 0$ be suitably small in terms of $\gamma$, and for all $N \in \mathbb{N}$ and $L \in \mathbb{R}$ satisfying $L \leqslant N$
let\[ W((x_n)_{n=1}^\infty,L,N): = \vert \{ n \leqslant N: x_n \in [Y,Y+L/N] \, \text{mod } 1 \}\vert.\] Then, if $L$ is in the range $N^{1 - \gamma + \varepsilon} < L \leqslant N$, we have 
\begin{equation}
\label{moment estimate}
\mathbb{E} W((x_n)_{n=1}^\infty,L,N)^k = L^k(1+O_{\varepsilon}(kN^{-\varepsilon/2})).
\end{equation} Furthermore, for all continuous functions $g:\mathbb{R}^{k-1} \longrightarrow [0,1]$ and for all $L$ in the range $N^{1 - \gamma + \varepsilon} < L < N^{1- \varepsilon}$,
\begin{equation}
\label{correlation estimate}
 R_k((x_n)_{n=1}^\infty, L,N,g) = (1 + o_{g,\varepsilon,k}(1))L^{k-1}\int g(\mathbf{w}) \, d\mathbf{w}
 \end{equation} as $N\rightarrow \infty$, where the error term is independent of the choice of parameters $L$. 
\\ 
(b) If $(a_n)_{n=1}^\infty$ 
is a strictly increasing sequence of positive integers, for almost every $\alpha \in [0,1]$, for all $N \in \mathbb{N}$ and for all $L \in \mathbb{R}$ in the range $N^{1/2 + \varepsilon} < L \leqslant N$,  
the sequence $$(\alpha a_n \text{ mod }1)_{n=1}^{\infty}$$
satisfies estimates (\ref{moment estimate}) and (\ref{correlation estimate}). 
\end{Lemma}
\noindent \emph{Remark}: Part (b) of the lemma proves our earlier assertion (\ref{eq: consequence of discrepancy}). 
\begin{proof}
Fix a small $\varepsilon>0$ throughout this proof.
For part (a), the proof of (\ref{moment estimate}) is trivial. Indeed, by the discrepancy estimate we have \[ \vert \{n \leqslant N: x_n \in [Y,Y+L/N] \, \text{mod } 1\}\vert = L + O(ND_N) = L + O_{\varepsilon}(N^{1 - \gamma + \varepsilon/2}) = L(1+O_{\varepsilon}(N^{- \varepsilon/2})).\] Raising to the $k^{th}$ power and averaging over $Y$, we obtain (\ref{moment estimate}). 

To prove the correlation estimate (\ref{correlation estimate}), by approximating the function $g$ by step functions 
we see that it is enough to prove it
in the case when $g$ is the indicator function 
of a box \[[s_1,t_1] \times \dots \times [s_{k-1}, t_{k-1}].\] 
We may assume without loss of generality that $N$ is large enough 
so that $(t_i - s_i)L/N \leqslant 1$ for all $i \leqslant {k-1}$. 
(This is why it is important for the correlation estimate 
to preclude the case $L = N$.) Then, fixing $n_k$, 
we see that $R_k((x_n)_{n=1}^\infty, L,N,g)$ 
counts the number of $x_{n_{k-1}}\neq x_{n_k}$ 
such that $x_{n_{k-1}} \in [s_{k-1}L/N + x_{n_k}, 
t_{k-1} L/N + x_{n_k}] \, \text{mod } 1$, 
times the number of $x_{n_{k-2}} \neq x_{n_{k-1}} ,x_{n_{k}}$ 
such that 
$x_{n_{k-2}} \in [s_{k-2}L/N + x_{n_{k-1}}, t_{k-2} L/N 
+ x_{n_{k-1}}]\, \text{mod } 1$, etc. 
By the discrepancy estimate, the total number of choices is 
\[ ((t_{k-1} - s_{k-1})L + O_{\varepsilon}(LN^{-\varepsilon/2})) 
\times ((t_{k-2} - s_{k-2})L + O_{\varepsilon}(LN^{-\varepsilon/2})) 
\times \dots \times ((t_1 - s_1)L + O_{\varepsilon}(LN^{-\varepsilon/2})).
\] 
Summing over all $n_k$ and then normalising by $1/N$, we have 
\[ R_k((x_n)_{n=1}^\infty, L, N, g) = 
\Big( \prod\limits_{i=1}^{k-1} (t_i - s_i) \Big) 
L^{k-1}(1 + O_{\varepsilon, g}(k N^{-\varepsilon/2}))\] as desired. 

This proves part (a) of the assertion. 
The remaining part follows by recalling that
a classical (and far more general) result of 
Erd\H{o}s and Koksma \cite[Thm. 2]{EK49} 
furnishes an upper bound on the discrepancy of 
$(\alpha a_n \text{ mod }1)_{n=1}^{\infty}$ 
of the quality $N^{-1/2 + \varepsilon}$,
for each fixed $\varepsilon >0$ 
and for almost every $\alpha \in [0,1]$. 
\end{proof}

Steinerberger \cite{S18} raised the question 
of whether `most sequences'\footnote{The precise meaning 
of the word `most' was left open for interpretation
by Steinerberger, and was already put in quotation marks
in the original paper.} have 
uniform pair correlations at some scale $0<\beta < 1$. 
The above part (b) answers Steinerberger's question in a strong sense.
Further, it seems worthwhile to record the following consequence.
To state it, recall that for any irrational $\alpha\in[0,1]$ there is a unique sequence of positive 
integers $\alpha_n \geq 1$ such that 
\begin{equation*}
\label{eq:contfrac}
\alpha = \lim_{n \rightarrow \infty}\cfrac{1}{\alpha_1 + \cfrac{1}{\alpha_2 + \cfrac{1}{\ddots + \cfrac{1}{\alpha_n}}}},
\end{equation*}
where $\alpha_n$ is called the $n^{th}$ {\em partial quotient of $\alpha$}. Writing the fraction \[\cfrac{1}{\alpha_1 + \cfrac{1}{\alpha_2 + \cfrac{1}{\ddots + \cfrac{1}{\alpha_n}}}}=: \frac{p_n}{q_n}\] in lowest terms, we call $q_n$ a \emph{convergent denominator of $\alpha$}. 
\begin{Corollary} 
\label{Corollary continued fraction}
Let $\alpha_n$ be the $n^{th}$ partial quotient
of the irrational number $\alpha\in[0,1]$. Given $N\geqslant 1$, let $i(N)$
be such that the convergent denominator $q_{i(N)}$
of $\alpha$ satisfies 
$q_{i(N)} \leqslant N < q_{i(N)+1}$.
If for each $\varepsilon>0$ we have $$ A_{N}(\alpha)
=\sum_{j \leqslant i(N)} \alpha_{j} \ll N^{\varepsilon},$$
then the Kronecker sequence
$(\alpha n)_{n=1}^\infty$ satisfies \eqref{correlation estimate},
for any $k\geqslant 2$ and for any scale $L$ 
such that $N^{\delta} <L < N$ for any fixed $\delta \in (0,1)$
(the $o(1)$ term in \eqref{correlation estimate}
then also depends on $\delta$).
\end{Corollary}

\begin{proof}
As $D_N((\alpha n \, \text{mod } 1)_{n=1}^\infty) 
\ll_{\varepsilon} A_N (\alpha)$, 
cf. \cite[Eq. (3.18)]{KN12}, Lemma \ref{lem: discrepancy}
completes the proof.
\end{proof}
It is well known 
(also with a higher-dimensional generalisations due to Beck \cite{B94})
that for almost every $\alpha\in[0,1]$ 
the discrepancy of the Kronecker sequence
is $\ll (\log N)^{1+\varepsilon}$, for each $\varepsilon >0$.
Thus the above corollary generalizes and sharpens 
a result of Skill and Wei{\ss} \cite{WS19} stating
that $(\phi n)_{n=1}^\infty$ has Poissonian pair correlations 
on each scale $\beta < 1$ where
$\phi$ denotes the Golden ratio $\frac{\sqrt{5}+1}{2}$.
Further, the condition on $A_N$ is known to be true for algebraic $\alpha$,
due to Roth's famous approximation theorem.

Finally we remark that, if one so wished, one could readily replace the
$N^{\varepsilon}$-terms in Corollary \ref{Corollary continued fraction} by appropriate powers of logarithms.\\

\section{The Rudnick-Sarnak obstruction}
\label{Appendix C}
In this final appendix, we detail an obstruction 
to studying the higher order correlation function $R_k(\alpha,L,N,g)$. 
This obstruction is a generalisation of a fundamental observation of 
Rudnick--Sarnak \cite[Section 4]{RS98}, which those authors made in the context of constant $L$ and
for the triple correlation function of $(\alpha n^2 \, \text{mod } 1)_{n=1}^\infty$.
We address the more general situation of
sequences of the shape $(\alpha n^d \, \text{mod } 1)_{n=1}^\infty$,
where $d\geqslant 2$ is a fixed integer. We are also interested in identifying the full range of $L$ in which the obstruction persists.

To this end, for a compactly supported smooth test function 
$g:\mathbb{R}^{k-1} \rightarrow \mathbb{R}$, we define the correlation function $R_{k}^{d}(\alpha,L,N, g)$ by \[R_{k}^{d}(\alpha,L,N, g): = \frac{1}{N}\sum\limits_{\substack{1 \leqslant x_1, \dots, x_k \leqslant N \\
\text{distinct}}} g\Big(\frac{N}{L} 
\{ \alpha (x_1^d - x_2^d)\}_{\mathrm{sgn}}, \dots, \frac{N}{L}
\{ \alpha (x_{k-1}^d - x_k^d)\}_{\mathrm{sgn}}\Big),\] where \[ \{ \cdot \}_{\mathrm{sgn}}:\mathbb{R} \longrightarrow (-1/2,1/2]\] denotes the signed distance to the nearest integer. Rudnick--Sarnak's approach to pair correlations, like in Appendix \ref{Appendix A}, involves showing that \[ \int\limits_0^1 \Big(R_2^d(\alpha,L,N,g) - L \frac{N-1}{N}\widehat{g}(0)\Big)^2 \, d\alpha = o_g(L^2)\] as $N \rightarrow \infty$. Therefore, in order to make a similar approach work for the 
higher $k$-point correlation functions, one would need 
\begin{equation}
\label{L2 convergence}\int\limits_{0}^{1} 
\Big(
R_k^d(\alpha,L,N,g)
- L^{k-1} \frac{(N)_{k}}{N^k}\widehat{g}(\mathbf{0})
\Big)^2 \, d\alpha = o_{d,k,g}(L^{2(k-1)})
\end{equation} 
as $N \rightarrow \infty$, where $\mathbf{0}$ is the zero-vector in $\mathbb{R}^{k-1}$
and $(N)_{k}=N(N-1)\ldots(N-k+1)$ abbreviates the $k^{th}$ falling factorial. Our purpose here is to show that for certain ranges of $d$, $k$ and $L$, equation (\ref{L2 convergence}) cannot hold. \\

Indeed, by applying the Poisson summation formula one may expand
$ R_k^d(\alpha,L,N,g) $ into a Fourier series 
\[ R_k^d(\alpha,L,N,g) = 
\sum\limits_{\ell \in \mathbb{Z}} 
c_{k,\ell}^d(L,N,g) e(\ell \alpha),\] 
with certain Fourier coefficients $c_{k,\ell}^d(L,N,g)$.
One may compute $c_{k,\ell}^d(L,N,g)$ explicitly. For a given vector $\mathbf{a}\in \mathbb{Z}^{k-1}$, let
$S_{\mathbf{a},k,\ell}^d(N)$ denote
the set of integer vectors $\mathbf{x}\in \mathbb{N}^{k}$,
with distinct components $1\leqslant x_i \leqslant N$, satisfying 
the Diophantine equation 
$$
a_1 (x_1^d - x_2^d)+\ldots + a_{k-1}(x_{k-1}^d-x_k^d)=\ell.
$$
One readily verifies that $ c_{k,\ell}^d(L,N,g)$ is of the special form 
\begin{equation}
\label{formula for c}
c_{k,\ell}^d(L,N,g) = 
\frac{L^{k-1}}{N^{k}}\sum_{\mathbf{a} \in \mathbb{Z}^{k-1}} 
\vert S_{\mathbf{a},k,\ell}^d(N) \vert \,\,
\widehat{g}\Big(\frac{L}{N} \mathbf{a}\Big).
\end{equation}From Parseval, we conclude that \[\int_0^1\Big(
R_k^d(\alpha,L,N,g)
- L^{k-1} \frac{(N)_{k}}{N^k}\widehat{g}(\mathbf{0})
\Big)^2 \, d\alpha = \sum\limits_{ \ell \neq 0} \vert c_{k, \ell}^d(L,N,g)\vert^2.\] 

Now let $\rho = N^{d+1 + \varepsilon}/L$. 
We observe that if $\vert \ell \vert \geqslant \rho $ and $S_{\mathbf{a},k,\ell}^d(N) \neq 0$ 
then  $\Vert \mathbf{a}\Vert_\infty \geqslant N^{1+\varepsilon}/L(k-1)$. Therefore, from the rapid decay of $\widehat{g}$ and the formula (\ref{formula for c}), we conclude that
$\vert c_{k,\ell}^d(L,N,g) \vert = O_{\varepsilon,g,k,K}(N^{-K})$ for such $\ell$. Hence 
\[\sum\limits_{\ell \neq 0} \vert c_{k,\ell}^d(L,N,g)\vert ^2  = \sum\limits_{0 < \vert \ell\vert \leqslant \rho} 
\vert c_{k,\ell}^d(L,N,g)\vert^2 
+ O_{\varepsilon,g,k,K}(N^{-K}).\]
To estimate the right hand side, we note that 
by the Cauchy--Schwarz inequality
$$
\rho \sum\limits_{ 0 < \vert \ell \vert \leqslant \rho} 
\vert c_{k,\ell}^d(L,N,g)\vert^2
\geqslant \bigg\vert \sum_{  0< \vert \ell \vert \leqslant \rho} 
c_{k,\ell}^d(L,N,g) \bigg\vert^2.
$$
The sum inside the absolute value on the right-hand side equals,
up to a term of size $O_{\varepsilon,g,k,K}(N^{-K})$, the quantity
\[
\sum\limits_{ \ell \in \mathbb{Z}} 
c_{k,\ell}^d(L,N,g) e(0) - L^{k-1} \frac{(N)_{k}}{N^k}\widehat{g}(\mathbf{0}),\] which is \[R_k^d(0,L,N,g) + O_{g}(L^{k-1}).\] Assuming that $g(\mathbf{0}) \neq 0$ and $L/N = o(1)$ as $N\rightarrow \infty$, this is equal to \[N^{k-1}(1+o(1))g(\mathbf{0})\] as $N\rightarrow \infty$.
Now, by combining these considerations, we conclude that 
\begin{align*}
\int\limits_{0}^{1} \Big(R_k^d(\alpha,L,N,g) - L^{k-1} \frac{(N)_k}{N^k} \widehat{g}(\mathbf{0})\Big)^2 
\, d\alpha
\gg_{\varepsilon,g,k} \frac{1}{\rho} N^{2(k-1)} \vert g(\mathbf{0})\vert^2
= L N^{2k - d-3 - \varepsilon} \vert g(\mathbf{0})\vert^2.
\end{align*}
If (\ref{L2 convergence}) is to hold for all functions $g$, 
for each $\varepsilon > 0$ we must have 
$$
L\gg_{g,\varepsilon,k} N^{\frac{2k-d-3-\varepsilon}{2k-3}} 
= N^{1 - \frac{d+\varepsilon}{2k-3}}.
$$ In particular, for $k=3$ and $d=2$ the convergence (\ref{L2 convergence}) fails unless
$L\gg N^{\frac{6-5-\varepsilon}{6-3}}= N^{\frac{1-\varepsilon}{3}}$. This justifies the statements we made in the introduction to the effect that the Rudnick--Sarnak obstruction for triple correlations extends throughout the range $L<N^{1/3}$. 

We also note, however, that as soon as \[ d > 2k-3\] there is no such obstruction (for $L$ constant in terms of $N$).\\

The failure of (\ref{L2 convergence}) is an artefact of the integrand having a large spike when $\alpha \approx 0$ (and more generally the integrand has large spikes when $\alpha$ is very well-approximated by rationals with small denominators). One wonders whether $L^2$-convergence (\ref{L2 convergence}) can be recovered by restricting to the `minor arcs', but this also appears to be a difficult problem.

\bibliography{tripbib}

\begin{thebibliography}{10}

\bibitem{AL16}
C.~Aistleitner and G.~Larcher.
\newblock Metric results on the discrepancy of sequences $(a_{n} \alpha)_{n
  \geq 1}$ modulo one for integer sequences $(a_{n})_{n \geq 1}$ of polynomial
  growth.
\newblock {\em Mathematika}, 62(2):478--491, 2016.

\bibitem{ALL17}
C.~Aistleitner, G.~Larcher, and M.~Lewko.
\newblock Additive energy and the {H}ausdorff dimension of the exceptional set
  in metric pair correlation problems.
\newblock {\em Israel J. Math.}, 222(1):463--485, 2017.

\bibitem{B94}
J.~Beck.
\newblock {Probabilistic diophantine approximation, I. Kronecker sequences}.
\newblock {\em Ann. of Math.}, 140(2):449--502, 1994.

\bibitem{BT77}
M.~V. Berry and M.~Tabor.
\newblock Level clustering in the regular spectrum.
\newblock {\em Proc. Royal Soc. London. A. Math. Phy. Sci.},
  356(1686):375--394, 1977.

\bibitem{BZ00}
F.~Boca and A.~Zaharescu.
\newblock Pair correlation of values of rational functions (mod $p$).
\newblock {\em Duke Math. J.}, 105(2):267--307, 2000.

\bibitem{Ch49}
K.-L. Chung.
\newblock An estimate concerning the {K}olmogoroff limit distribution.
\newblock {\em Trans. Amer. Math. Soc.}, 67(1):36--50, 1949.

\bibitem{BMV15}
D.~El-Baz, J.~Marklof, and I.~Vinogradov.
\newblock The two-point correlation function of the fractional parts of
  {$\sqrt{n}$} is {P}oisson.
\newblock {\em Proc. Amer. Math. Soc.}, 143(7):2815--2828, 2015.

\bibitem{EM04}
N.~Elkies and C.~McMullen.
\newblock Gaps in $\sqrt{n}$ mod 1 and ergodic theory.
\newblock {\em Duke Math. J.}, 123(1):95--139, 2004.

\bibitem{EK49}
P.~Erd\H{o}s and J.~Koksma.
\newblock On the uniform distribution modulo 1 of sequences $(f(n,
  \vartheta))$.
\newblock {\em Nederl. Akad. Wetensch., Proc. 52}, pages 851--854, 1949.

\bibitem{EMM05}
A.~Eskin, G.~Margulis, and S.~Mozes.
\newblock Quadratic forms of signature (2, 2) and eigenvalue spacings on
  rectangular 2-tori.
\newblock {\em Ann. of Math.}, 161(2):679--725, 2005.

\bibitem{Ha98}
G.~Harman.
\newblock {\em Metric Number Theory, London Mathematical Society Monographs}.
\newblock Clarendon Press (Oxford), 1998.

\bibitem{HB10}
D.~R. Heath-Brown.
\newblock Pair correlation for fractional parts of $\alpha n^2 $.
\newblock {\em Math. Proc. Cambridge Phil. Soc.}, 148(3):385--407, 2010.

\bibitem{Iw04}
H.~Iwaniec and E.~Kowalski.
\newblock {\em Analytic number theory}, volume~53.
\newblock American Mathematical Soc., 2004.

\bibitem{K24}
A.~Khintchine.
\newblock {\"{U}}ber einen {S}atz der {W}ahrscheinlichkeitsrechnung.
\newblock {\em Fund. Math.}, 6(1):9--20, 1924.

\bibitem{KN12}
L.~Kuipers and H.~Niederreiter.
\newblock {\em Uniform distribution of sequences}.
\newblock Courier Corporation, 2012.

\bibitem{Lu20}
C.~Lutsko.
\newblock Long-range correlations of sequence modulo 1.
\newblock \emph{arXiv:2007.09292}.

\bibitem{M01}
J.~Marklof.
\newblock The {B}erry-{T}abor conjecture.
\newblock In {\em Europ. Congr. Math.}, pages 421--427. Springer, 2001.

\bibitem{M03}
J.~Marklof.
\newblock Pair correlation densities of inhomogeneous quadratic forms.
\newblock {\em Ann. of Math.}, pages 419--471, 2003.

\bibitem{M94}
H.~Montgomery.
\newblock {\em Ten lectures on the interface between analytic number theory and
  harmonic analysis}.
\newblock Number~84. American Mathematical Soc., 1994.

\bibitem{R08}
Z.~Rudnick.
\newblock Quantum chaos?
\newblock {\em Notices of the AMS}, 55(1):32--34, 2008.

\bibitem{Ru18}
Z.~Rudnick.
\newblock A metric theory of minimal gaps.
\newblock {\em Mathematika}, 64(3):628--636, 2018.

\bibitem{KR99}
Z.~Rudnick and P.~Kurlberg.
\newblock The distribution of spacings between quadratic residues.
\newblock {\em Duke J. Math.}, 100:211--242, 1999.

\bibitem{RS98}
Z.~Rudnick and P.~Sarnak.
\newblock The pair correlation function of fractional parts of polynomials.
\newblock {\em Comm. Math. Phy.}, 194(1):61--70, 1998.

\bibitem{RSZ01}
Z.~Rudnick, P.~Sarnak, and A.~Zaharescu.
\newblock The distribution of spacings between the fractional parts of $n^2
  \alpha$.
\newblock {\em Invent. Math.}, 145(1):37--57, 2001.

\bibitem{RZ02}
Z.~Rudnick and A.~Zaharescu.
\newblock The distribution of spacings between fractional parts of lacunary
  sequences.
\newblock {\em Forum Math.}, 14(5):691--712, 2002.

\bibitem{Sa96}
P.~Sarnak.
\newblock Values at integers of binary quadratic forms, harmonic analysis and
  number theory ({M}ontreal, pq, 1996), 181-203.
\newblock In {\em CMS Conf. Proc}, volume~21.

\bibitem{S18}
S.~Steinerberger.
\newblock Poissonian pair correlation and discrepancy.
\newblock {\em Indag. Math.}, 29(5):1167--1178, 2018.

\bibitem{WS19}
C.~Wei{\ss} and T.~Skill.
\newblock Sequences with almost poissonian pair correlations.
\newblock \emph{arXiv:1905.02760}.

\bibitem{W16}
H.~Weyl.
\newblock {\"U}ber die {g}leichverteilung von {z}ahlen mod. {e}ins.
\newblock {\em Math. Ann.}, 77(3):313--352, 1916.

\bibitem{Za03}
Alexandru Zaharescu.
\newblock Correlation of fractional parts of {$n^2\alpha$}.
\newblock {\em Forum Math.}, 15(1):1--21, 2003.

\end{thebibliography}
\end{document}